\documentclass[11pt]{amsart}
\usepackage{amsmath,amscd,amssymb}
\usepackage[mathscr]{eucal}
\usepackage{amsmath}
\usepackage{enumerate}
\usepackage{color}

\DeclareMathOperator{\arccot}{arccot}

 \newtheorem{thm}{Theorem}[section]
 \newtheorem{cor}[thm]{Corollary}
 \newtheorem{lemma}[thm]{Lemma}
 \newtheorem{prop}[thm]{Proposition}

 \theoremstyle{definition}

 \theoremstyle{remark}
 \newtheorem{rem}[thm]{Remark}

\newtheorem{exa}[thm]{Example}

\numberwithin{equation}{section}

\def\R{\mathbb{R}}
\def\C{\mathbb{C}}
\def\N{\mathbb{N}}

\def\a{\alpha}

\def\ep{\varepsilon}
\def\l{\lambda}

\def\Sect#1{\operatorname{Sect}#1}

\def\H{\mathcal{H}}

\def\dom{\operatorname{dom}}

\def\msm{(1+m^{-1})s}
\title[Rational approximation of semigroups]{Rational approximation of holomorphic semigroups revisited}

\author{Charles Batty}
\address{St.\ John's College, University of Oxford, Oxford OX1 3JP, England}
\email{charles.batty@sjc.ox.ac.uk}

\author{Alexander Gomilko}
\address{Faculty of Mathematics and Computer Science\\
Nicolas Copernicus University\\
Chopin Street 12/18\\
87-100 Toru\'n, Poland}
\email{alex@gomilko.com}

\author{Yuri Tomilov}
\address{
Institute of Mathematics\\
Polish Academy of Sciences\\
\'Sniadeckich 8\\
00-656 Warsaw, Poland \\
and \\
Faculty of Mathematics and Computer Science\\
Nicolas Copernicus University\\
Chopin Street 12/18\\
87-100 Toru\'n, Poland
}
\email{ytomilov@impan.pl}

\begin{document}


\def\today{\number\day \space\ifcase\month\or
 January\or February\or March\or April\or May\or June\or
 July\or August\or September\or October\or November\or December\fi
 \space \number\year}
\date{\today}

\thanks{This work was partially supported by an NCN grant UMO-
2023/49/B/ST1/01961.}

\begin{abstract}
Using a recently developed $\mathcal H$-calculus
we propose a unified approach to the study
of rational approximations of holomorphic semigroups
on Banach spaces. We provide unified and simple proofs to a number of
basic results on semigroup approximations and substantially improve some of them.
We show that many of our estimates are essentially optimal, thus complementing the existing literature.
\end{abstract}

\maketitle
\section{Introduction}

The theory of rational appoximations of operator semigroups is an established area
intimately related to the numerical analysis of PDEs,
and also of value in probability theory and function theory.
The paper offers a new, functional calculus approach to the study
of rational approximations
arising from discretisation of an abstract Cauchy problem
\begin{equation}\label{ac}
x'(s)=-Ax(s), \qquad s \in [0,t], \quad x(0)=x_0,
\end{equation}
on a Banach space $X$, 
where $A$ belongs to the set of sectorial operators ${\rm Sect}(\theta)$
on $X$ of angle $\theta \in [0,\pi/2)$. Since $-A$ generates a
bounded holomorphic $C_0$-semigroup $(e^{-tA})_{t \ge 0}$,
mild solutions to \eqref{ac}
are given by $x(s)=e^{-sA} x_0, x_0 \in X$.
 When considering a division of $[0,t]$ into equal time steps $t/n$
and discretising \eqref{ac} accordingly, one approximates
a bounded holomorphic $C_0$-semigroup $(e^{-tA})_{t \ge 0}$
by $\{r^n(tA/n): n \in \mathbb N, t \ge 0\}$ as propagator, where $r$ is an appropriate rational function.

There are two main problems arising in the study of rational approximations $r^n(tA/n)$
and their generalizations. The first is to find the best possible operator norm bounds for 
$r^n(tA/n)$  and, ideally, to show  that $\sup_{n \ge 1, t \ge 0}\|r^n(tA/n)\|<\infty$, i.e.\ to prove stability of $r^n(tA/n)$. If the stability does not hold, then
the convergence of $r^n(tA/n)x$ to $e^{-tA}x$ as $n \to \infty$ for all $t \ge 0$ may not take place
for all $x \in X$. However  sharp bounds for $\|r^n(tA/n)\|$
help to deal with the second problem
of obtaining the convergence of $e^{-tA}x -r^n(tA/n)x$ for, at least, sufficiently regular initial data $x$, 
and to equip this convergence with optimal approximation rates.

The two problems above lead to the two corresponding classes of rational functions which are basic in numerical analysis.
To study stability of rational approximations of $e^{-tA}$ it is natural and necessary (even if $X=\mathbb C$)
to restrict attention to $\mathcal A(\psi)$-stable rational functions $r$, that is  $r$ satisfying
$|r(\lambda)|\le 1$ for $\lambda$ from  $\Sigma_\psi:=\{\lambda \in \mathbb C: |{\rm arg}\, \lambda|<\psi\}$
with $\psi \in (0,\pi/2]$. If $\psi=\pi/2$, then $r$ is said to be $\mathcal A$-stable.
The study of the approximation rates requires dealing with rational functions which approximate
$e^{-tz}$ at zero well enough: a rational function $r$ is said to be
an approximation of order $q\in \N$ to the exponential $e^{-z}$  if 
\begin{equation}\label{approxim}
e^{-z}-r(z)=\mbox{O}(|z|^{q+1}),\qquad z\to 0,
\end{equation}
or, in other words, $r^{(k)}(0)=(-1)^k$ for all $0 \le k \le q$.

Note that the approximation problems for \eqref{ac} are often considered
with variable step sizes $\mathcal K_n=\{k_j\}_{j=1}^n \subset (0,\infty)$,
$k_1+ \cdots +k_n=t$, and it is of interest to obtain
norm bounds for $r(k_1 A) \cdots r (k_n A)$
independent of $t$. While approximation rates can be studied in such a setting as well, and 
results similar to the case of constant steps can also be obtained (see \cite[p.\ 95, Eq.\ (7)]{Palencia}),
most of the literature has concentrated on approximation rates for constant step sizes, and we also follow this line of research.
 
Rational approximations of $(e^{-tA})_{t \ge 0}$ can be treated by variety
of means. One efficient approach, not requiring precise information
on the structure of $A$, relies on employing functional calculi for $A$.
In \cite{HK} Hersh and Kato proposed an approach to the study of rational approximations of bounded $C_0$-semigroups on Banach spaces based on the Hille-Phillips calculus, and thus on function estimates in the algebra of Laplace transforms of bounded measures on $\R_+$. This direction was further developed by Brenner and Thomee in \cite{BT} and a number of subsequent papers, see \cite{BakaevO}, \cite{EgertR},  \cite{GT}, \cite{GKT},\cite{Jara1}, and \cite{Kovacs} as examples, and many relevant approximation bounds got their optimal form. 
However, in the setting of (sectorially) bounded holomorphic semigroups
the estimates offered by the Hille-Phillips calculus appear to be rather crude. 
Thus, the existing arguments were mostly based on case by case studies
via the holomorphic Riesz-Dunford calculus. The arguments relied on appropriate choices of integration contours
and estimates for Riesz-Dunford integrals. They were rather technical,
and had to be adjusted to each problem accordingly.

The aim of the paper is twofold. First, using a new functional calculus, we provide a unified approach  
to a variety of norm estimates for rational approximations of holomorphic semigroups obtained in the literature, and as a byproduct,  improve some of them. Given $A \in {\rm Sect}(\theta), \theta \in [0,\pi/2)$, 
we use a bounded homomorphism from an appropriate algebra
of holomorphic functions $\mathcal H_\psi$ on a sector $\Sigma_\psi$,
 $\psi \in (\theta, \pi/2]$,  constructed in \cite{BaGoTo_J}
and called $\mathcal H$-calculus, see Section \ref{sec1} for more details.
The $\mathcal H$-calculus is compatible with the standard holomorphic functional calculus,
and since all the relevant rational functions arising in our considerations
belong to $\mathcal H_\psi$, it reduces the study of rational approximations
to suitable $\mathcal H_\psi$-norm estimates. In this way we develop a technique  for the   $\mathcal H_\psi$-norm estimates of rational functions based on fine control of their oscillations along the rays and sharp asymptotic estimates of approximation errors in a sector around zero.

In particular, using the $\mathcal H$-calculus, we prove two stability estimates for 
$\mathcal A(\psi)$-stable rational approximations $r$ and variable step sizes $\mathcal K_n= (k_j)_{j=1}^n$.
The first one, originally due to Bakaev \cite[Theorem 2.1]{Bakaev96} (see also \cite[Corollary 1.2]{Palencia1}),
provides a general operator norm estimate for $P_{{\mathcal K}_n,r}(A)=r(k_1A) \cdots r(k_nA)$
in terms of the maximum ratio of $k_j$.
If $\max_j k_j \le C \min_j k_j$ for some $C>0$, in particular if all $k_j$ are the same, then the estimate shows stability of $P_{{\mathcal K_n}, r}(A)$, and, if $r$ approximate $e^{-z}$ with certain order, this leads to convergence of $P_{{\mathcal K}_n,r}$ on $X$ as remarked in \cite[p.\,95]{Palencia}. 
The result was preceded by important stability estimates in \cite{Crou} and \cite{Palencia1} 
yielding stability for the family $\{r^n(tA/n): n \in \mathbb N, t \ge 0\}$. 
Our second estimate,  which is in a weaker form due to Bakaev \cite{Bakaev} and Palencia \cite[Theorem 1]{Palencia1} (see also \cite{Bakaev98} where the result is reproved),  shows stability of $P_{{\mathcal K}_n,r}(A)$
if $|r(\infty)|<1$. Both authors assume that $r$ is a first order approximation to the exponential, 
their proofs are rather involved and they rely essentially on this assumption.
The $\mathcal H$-calculus approach is direct, and it allows us to show that the approximation order of $r$ is irrelevant, and thus the result holds in greater generality.
These estimates serve as illustrations of our technique, and other applications such as the study of multi-step
approximations as in \cite{Palencia} could probably be handled similarly. We leave this to the interested reader.
For an abstract approach to variable stepsize approximation and other related results one may consult \cite[Section II.8]{Bakaev_b}.

As the second aim of the paper (related to the first one),
we present a new, $\mathcal H$-calculus approach for obtaining 
rational approximation rates for holomorphic semigroups and strengthen 
several well-known statements in the literature. 
Let us first recall 
several basic bounds on approximation rates in the setting of holomorphic semigroups.
\begin{thm} \label{basic}
Let $r$ be an $\mathcal{A}(\psi)$-stable rational approximation of order $q$ to the exponential, and let  
$A \in \operatorname{Sect}(\varphi)$ for some $\varphi\in [0,\theta)$ and $\theta \in (0,\psi)$.
Then there exists a constant $C = C(\theta)$ such that the following hold for all $s\in [0,q]$, $n\in\N$ 
and $t>0$:
  \begin{enumerate}[\rm(i)]
\item If $|r(\infty)|=1$, then 
\begin{equation}\label{oldr}
\|e^{-tA}x - r^n(tA/n)x\| \le CM_\theta(A) t^s n^{-s} \|A^sx\|, \qquad x \in \operatorname{dom}(A^s).
\end{equation}
\item If $|r(\infty)|<1$, then 
\begin{equation}\label{oldr1}
\|e^{-tA} - r^n(tA/n)\| \le CM_\theta(A)/n^q.
\end{equation}
\end{enumerate}
\end{thm}
Theorem \ref{basic}(i) was apparently first remarked in \cite[p.\ 694]{BT} for $s=q$
as a consequence of arguments given by Le Roux \cite{Roux}, and it was repeated in \cite[Theorems 3 and 4]{Crou} again for $s=q$ (mentioning also \cite{Fujita} and \cite{Sobolevskii}).
See also \cite[Theorem 9.2 and Theorem 9.3]{Thomee_b}, where Theorem \ref{basic} with $s=q$ in (i)
was given with a full proof.
Both parts (i) and (ii) of Theorem \ref{basic}, for integer $s$ and invertible $A$ were proved in \cite[Theorem 4.2]{Larsson} and \cite[Theorem 4.4]{Larsson}, while in \cite[p.\ 95]{Palencia} it was noted that that the argument in \cite{Larsson} can be adjusted to get rid of invertibility of $A$ and to consider even variable step approximations, see also \cite[Theorem 2.1]{Hansbo}. 
(Concerning lifting the invertibility of $A$ see the remark following \cite[Theorem 5]{Crou}).
Generalisations of Theorem \ref{basic} for invertible $A$, integer $s$ and variable time steps were obtained in \cite[Theorem 2.1 and Theorem 2.6]{Yan}. (See also \cite{Yan} and \cite{Yan1} for other related results on rational approximation.)
Note \cite[Theorem 2.2 and Theorem 2.5]{Flory}, where Theorem \ref{basic} for integer $s$ was obtained
assuming that $A$ admits a bounded $H^\infty$-calculus. Finally, Theorem \ref{basic} in the present general formulation was proved in \cite[Theorem 9.2.3 and Corollary 9.2.9]{Haase_b}.

In the form of Theorem \ref{basic} (with varying generality) the result on rational approximation
of holomorphic  semigroups has been functioning in the numerical analysis community since the $1980$s, and it became standard (e.g.\ in \cite{Jara1}). In this paper, we show that Theorem \ref{basic}(i) can, in fact, be substantially improved up to its optimal form. To this aim we introduce the notion of $\mathcal A(\psi,m)$-stability of a rational function. Let $\psi\in (0,\pi/2]$ and let $r$ be a non-constant $\mathcal A(\psi)$-stable rational function. Then $ r(\infty) = \lim_{z \to \infty} r(z)$ exists in $\C$,  $|r(\infty)| \le 1$, and
 there exist $a \in \C \setminus \{0\}$ and  $m\in \N$ such that
\begin{equation}\label{order}
r(z)=r(\infty)-\frac{a}{z^m}+\mbox{O}(|z|^{-(m+1)}),\qquad z\to\infty.
\end{equation}
If $r$ satisfies \eqref{order}, then we say that $r$ is \emph{$\mathcal A(\psi, m)$-stable}.
Observe that if $\epsilon > 0$ and $r_{\epsilon}(z) = r(z +\epsilon)$, then $r_\epsilon$ is $\mathcal A(\psi, m)$-stable,
with the same value of $a$ in \eqref{order}.
Moreover, if $r$ is $\mathcal A(\psi)$-stable, then
there exists $c_r>0$ such that
\begin{equation}\label{derivative}
|r'(z)|\le \frac{c_r}{(1+|z|)^2}, \qquad z \in \Sigma_\psi.
\end{equation}

Assuming that $r$ is an $\mathcal A(\psi,m)$-stable approximation of order $q$ to $e^{-z}$,
with $|r(\infty)|=1$, and setting for simplicity $t=1$, we prove in Theorem \ref{T1PHA} that if $s \in [0, mq/(m+1))$, then the approximation rates $n^{-s}$ in \eqref{oldr} can be replaced by strictly
sharper rates $n^{-s(1+1/m)}$. If $s \in [mq/(m+1), q+1]$, we improve again the rates $n^{-s}$ given by \eqref{oldr} and replace them with the rates $n^{-q}$. The latter rates are in fact the fastest possible when $s$ varies through the whole of  $[0,q+1]$ as we show in in Theorem  \ref{optimal_t}. 
Our results have an especially transparent form if $r$ is one of the diagonal Pad\'e approximations to $e^{-z}$. Then  $m=1$, and Corollary \ref{Pade} yields the maximal  convergence rate $n^{-q}$ on quite a large interval $[q/2,q+1]$ containing, in particular, many integers for large $q$.
For $s \in [q,q+1]$ we do not get any improvements for approximation rates thus observing 
a ``saturation" effect, and this interval for $s$ seems not to have been discussed in the literature.
These phenomena can be explained by the fact that the $\mathcal H_\psi$-norm estimates for $r^n(\cdot/n)$
also do not improve when $s \in [q, q+1]$.

If $|r(\infty)|<1$, then we show that \eqref{oldr1} holds in a slightly generalised form 
with the right hand side replaced by $\|A^s x\|/n^q$ for $s \in [0,q+1]$ and all 
$x \in \operatorname{dom}(A^s)$.  One gets \eqref{oldr1} if $s=0$,
but, on the other hand, $\|A^s x\|$ and $\|x\|$ are, in general, not comparable for $x \in \operatorname{dom}(A^s)$ if $A$ is not invertible.

An important feature of our abstract approach is that we are able to show that
the obtained approximation rates are optimal in a strong sense. Namely, by an operator-theoretical construction of independent interest, we prove that  if $r$ is
an $\mathcal A(\psi,m)$-stable rational approximation of order $q$ to the exponential then
for all $\theta \in [0, \pi/2)$ and $s \in [0, q+1]$, there exist a Banach space $X$, an
operator $A \in \operatorname{Sect}(\theta)$ on $X$, and a vector $x \in \operatorname{dom}(A^s)$ 
such that the lower bounds for  
$\|e^{-tA}x- r(A/n)^nx \|$
 match the approximation rates obtained in Theorem \ref{T1PHA} up to a constant depending on $\theta$
and a sequence vanishing at infinity.   Optimality issues are very
natural in approximation theory, and they have been addressed thoroughly for rational approximations
of bounded $C_0$-semigroups, but we have not been able to find any similar results in the literature. A typical result would obtain optimality only in a certain (operator norm) sense, see e.g.\ \cite[Theorem 4.4]{BT},
 \cite[Chapter 5, Theorems 4.3 and 4.4]{BTW}, \cite[p. 285]{Kovacs}, \cite[Corollary 7.5]{GT} or \cite[Prop. 4.8]{GKT}. This was a starting point of our research which led to unexpected consequences discussed above.

 \subsection*{Conventions and notation}
In this paper, we will consider an individual rational function $r$.
It will be implicit that some functions and constants that are introduced will depend on $r$
and other parameters associated solely with $r$ (for example, the order $q$).
They may depend on other parameters which are independent of $r$, and 
we will try to make clear which parameters are involved by writing $C = C(n,s)$ or  $C_{n,s}$,
if the constants depend on $n$ and $s$.

For readers' convenience, we list here various functions derived from $r$.  The variables are $n \in \N$, $s \in \R_+=[0,\infty)$, $\gamma \in \R_+$, $z \in \C$. The following notation for relevant functions
will be used frequently:
\begin{itemize}
\item $r_n(z) = r^n(z/n)$, 
\item $r_{n,s}(z) = r_n(z)/z^s$, 
\item $\Delta_n(z) = e^{-z}-r_n(z)$, 
\item $\Delta_{n,s}(z) = \Delta_n(z)/z^s$. 
\end{itemize}
We will also denote
$\C_+:=\{z \in \mathbb C: {\rm Re}\, z >0\}$, $\mathbb{D}:=\{z \in \mathbb C: |z|<1\}$, 
 $\mathbb{D}_R:=\{z \in \mathbb C: |z|<R\}$, and  
$\Sigma_\theta:=\{z \in \mathbb C: |{\rm arg}\, z|<\theta\}$, $\theta \in (0,\pi)$.
The set of holomorphic functions on  a domain $\Omega \subset \mathbb C$ will be denoted by ${\rm Hol}(\Omega)$,
and $H^\infty(\Omega)$ will stand for the space of $f \in {\rm Hol}(\Omega)$ which are bounded in $\Omega$.

If $X$ is a Banach space, then the space of bounded linear operators on $X$ will be denoted by $L(X)$,
and $\Sect(\theta)$ will stand for the set of sectorial operators on $X$.

\section{The HP-calculus and the $\mathcal H$-calculus}\label{sec1}

Our considerations will be based on the so-called $\mathcal H$-functional calculus
constructed recently in \cite{BaGoTo_J}.
To put the $\mathcal H$-calculus into a proper context, we first
recall another functional caclulus, which was used in the study of rational approximations of semigroups, see  \cite{BT}.

Let ${\rm M}(\R_+)$ denote the Banach algebra of bounded Borel measures on $\mathbb R_+$, and let
$\mathcal{LM}(\C_+)$ be the Banach algebra of Laplace transforms $\mathcal L \mu$, where  $\mu \in {\rm M}(\R_+)$,
equipped with the norm  
\begin{equation}\label{hpn}
\|\mathcal L \mu\|_{\mathcal{LM}}:=\|\mu\|_{{\rm M}(\R_+)}.
\end{equation}
If $-A$ generates a bounded $C_0$-semigroup $(e^{-tA})_{t \ge 0}$ on a Banach space $X$,
and $M:=\sup_{t \ge 0}\|e^{-tA}\|$,
then the Hille-Phillips (HP-) calculus for $A$ is defined as the bounded homomorphism $\Psi_A: \mathcal{LM}\to
L(X)$
given by the operator version of the Laplace transform:
$$\Psi_A(\mathcal{L} \mu)x=\mathcal{L} \mu (A)x=\int_{0}^\infty e^{-tA}x\, d\mu(t), \qquad  x \in X, \quad \mathcal{L} \mu \in \mathcal{LM},$$  
so that 
\begin{equation}\label{hpn1}
\|\mathcal L \mu(A)\|\le M \|\mu\|_{{\rm M}(\R_+)}.
\end{equation}

The HP-calculus is one of the basic functional calculi used in the literature. However, it often leads to rather crude estimates and ignores specific aspects of important subclasses of the class of bounded $C_0$-semigroups, such as sectorially bounded holomorphic semigroups. In the studies of rational approximations for holomorphic semigroups, another (holomorphic) functional calculus based on Riesz-Dunford integrals is used frequently. However, in this case, the arguments depend on delicate choices of contours and involved estimates which depend on particular contexts.

To produce sharp approximation norm bounds,
finer calculi are required. One such calculus was
developed in \cite{BaGoTo_J}, and we recall some of its features now.
Let $H^1(\Sigma_\psi)$ stand for the Hardy space of functions
$f \in {\rm Hol}(\Sigma_\psi)$ such that
\[
\|f\|_{H^1(\Sigma_\psi)}:=\sup_{|\theta|<\psi}\,
\int_0^\infty
\left(|f(te^{i\theta})|+|f(te^{-i\theta})|\right)\,dt.
\]
It was shown in \cite[Theorem 4.6]{BaGoTo_J} that if $f \in H^1(\Sigma_\psi)$, then the boundary values
$f(te^{\pm i\psi})=\lim_{\varphi \to \pm \psi}f(te^{it\varphi})$ exist for a.e.\ $t$
and $\|f\|_{H^1(\Sigma_\psi)}=\|f\|_{L^1(\partial \Sigma_\psi)}$.

The Hardy-Sobolev space $\mathcal{H}_\psi$,
$\psi\in (0,\pi)$, is defined as the space of $f \in {\rm Hol}(\Sigma_\psi)$
such that
\[
\|f\|_{\mathcal H_{\psi,0}}:=
\|f'\|_{H^1(\Sigma_\psi)}<\infty.
\]
By \cite[Theorem 4.8]{BaGoTo_J},  if $f \in \mathcal H_\psi$ then $f \in H^\infty(\Sigma_\psi)$, $f$ extends continuously to $\partial \Sigma_\psi$, and there exists
$f(\infty)=\lim_{z \to \infty, z \in \Sigma_\psi} f(z)$. Note that  by the above,
\begin{equation}\label{normatt}
\|f\|_{\mathcal H_{\psi,0}}=\|f'\|_{L^1(\partial \Sigma_\psi)}
\end{equation}
Moreover, by \cite[Lemma 4.9]{BaGoTo_J}, $\mathcal H_\psi$ is a Banach algebra equipped with the norm $\|f\|_{\mathcal H_\psi}:=\|f\|_\infty+\|f\|_{\mathcal H_{\psi, 0}}$.
If $f\in {\rm Hol}(\Sigma_\theta)$, then
\begin{equation}\label{monotone}
\|f\|_{\mathcal{H}_\psi}\le \|f\|_{\mathcal{H}_\theta},\quad
f\in \mathcal{H}_\theta,\quad 0<\psi<\theta.
\end{equation}
The expression
\[
\|f\|^*_{\mathcal{H}_\psi}:=|f(\infty)|+\|f\|_{\H_{\psi,0}}, \qquad f \in \mathcal H_\psi,
\]
defines an equivalent norm on $\mathcal H_\psi$, which is not a Banach algebra norm.   In particular, \cite[Theorem 4.8(iii)]{BaGoTo_J} shows that, if $f \in \H_\psi$ and $f(\infty)=0$, then
\begin{equation} \label{11inf}
\|f\|_{\infty} \le \|f\|_{\H_\psi}.
\end{equation}

The Banach algebras $\H_{\psi}$ for $\psi \in (0,\pi)$ are isometrically isomorphic to each other, via fractional powers.
Let $\psi_1,\psi_2 \in (0,\pi)$ and $\gamma = \psi_1/\psi_2$.  If $f \in \H_{\psi_1}$ and $f_\gamma(z) = f(z^\gamma)$, it is easy to check that $f_\gamma\in\H_{\psi_2}$ and $\|f_\gamma\|_{\H_{\psi_2,0}} = \|f\|_{\H_{\psi_1,0}}$.   See \cite[Lemma 4.9]{BaGoTo_J}.

Moreover, it is instructive to observe that if $f \in \mathcal{LM}$, then (the restriction to $\Sigma_\psi$ of) $f$ belongs to $\mathcal H_\psi$ for any $\psi \in (0,\pi/2)$. Indeed, if $ f=\mathcal L \mu$ and $|\theta|< \psi$, then
\begin{align*}
\int_0^\infty |f'(te^{i\theta})|\, dt\le
\int_{\R_+}\int_0^\infty
e^{-t\tau \cos\theta}\,dt \,\tau |\mu|(d\tau)
=\frac{1}{\cos\theta}\int_{\R_+}
|\mu|(d\tau),
\end{align*}
so that
\begin{align}\label{lm-h}
\|f\|^*_{\mathcal{H}_\psi}\le |\mu(\{0\})|+\frac{2}{\cos\psi}\int_{(0,\infty)}
|\mu|(d\tau)
\le
\frac{2}{\cos\psi}\|f\|_{\mathcal{LM}}.
\end{align}

As shown in \cite{BaGoTo_J}, the algebras $\mathcal H_\psi$ allow one to define a functional calculus
for a sectorial operator $A$ on a Banach space $X$,  which takes into account the angle of sectoriality
of $A$ and leads to fine operator-norm estimates for functions of $A$.
More precisely, recall that the operator $A$ on $X$ is said to be sectorial of angle $\theta$
if $\sigma(A) \subset \overline{\Sigma}_\theta$, and for every $\psi \in (\theta, \pi)$
there exists $M_\psi(A)>0$ such that
\[
 M_\psi (A):=\sup \{\|\lambda R(\lambda, A)\|: \lambda \in \Sigma_{\pi-\psi}\}<\infty.
\]
As in the introduction, we let $\Sect(\theta)$ stand for the class of sectorial operators of angle $\theta$,
for $\psi \in [0,\pi)$ and  $\l \in \Sigma_{\pi-\psi}$ denote
 $\rho_\l(z): = (z+\l)^{-1}$. 
If $A\in\Sect(\theta)$, $\psi \in (\theta, \pi)$,
then by \cite[Theorem 8.6]{BaGoTo_J} there is a unique bounded homomorphism
 $\Psi_A: \mathcal H_\psi \to L(X)$, $ \Psi_A(f)=f(A)$,
such that $\Psi_A(\rho_\lambda)=(A+\lambda)^{-1}, \lambda \in \Sigma_{\pi-\psi}$,
called the $\mathcal H$-calculus for $A$.  
The calculus is given by
the formula:   
\begin{equation}\label{hformula_oper}
f(A)=f(\infty)-
\frac{1}{\pi}\int_0^\infty \left(f'(e^{it\psi}) + f'(e^{-it\psi})\right)
\arccot(A^\gamma/t^\gamma)\,dt,
\end{equation}
where the operator kernel $\arccot(A^\gamma/t^\gamma)$ is defined by a separate
formula involving the logarithmic function as explained in \cite[Section 8.2]{BaGoTo_J}. 
From \eqref{hformula_oper} it follows that
\begin{equation}\label{estimate}
\|f(A)\|\le |f(\infty)| + \frac{M_\psi(A)}{2} \|f\|_{\H_{\psi,0}} \le M_\psi(A)\|f\|_{\mathcal{H}_\psi}.
\end{equation}
Thus, if $f(\infty)=0$, 
\begin{equation} \label{est0}
\|f(A)\|\le \frac{M_\psi(A)}{2} \|f\|_{\H_{\psi,0}}.
\end{equation}
The formula \eqref{hformula_oper} appears to depend on $\psi$, but the definitions of $f(A)$ agree provided that $\psi \in (\theta,\pi)$ and $f$ is holomorphic on $\Sigma_\psi$. Recall that $A\in \Sect(\theta), \theta \in [0,\pi/2)$, if and only if $-A$ generates a (sectorially bounded) holomorphic $C_0$-semigroup on $X$,
so that the $\mathcal H$-calculus applies, in fact, to the negative generators of holomorphic semigroups.  

We emphasize two crucial properties of the $\mathcal H$-calculus, which are implicit
in \cite{BaGoTo_J} in view of a somewhat indirect approach to the $\mathcal H$-calculus developed there.
Note that by \cite[Theorems 5.10, 11.1]{BaGoTo_J}, the algebras $\mathcal H_\psi$ 
are invariant under shifts, i.e., if $f \in \mathcal H_\psi$, $\epsilon>0$, and $f_\epsilon(z) = f(z+\epsilon)$,
then $f_\epsilon \in \mathcal H_\psi$, and the map $f \mapsto f_{\ep}$ from $\mathcal H_\psi$ to $\mathcal H_\psi$ is bounded.
Thus, it is easy to see that if $A \in \Sect(\theta)$, then $A+\ep$ has a bounded $\mathcal H_\psi$-calculus given by
\[
\Phi_{A+\epsilon}(f) = \Phi_A(f_\epsilon).
\]
Since the $\mathcal H_\psi$-calculus for $A+\ep$ is unique, it follows that 
\begin{equation}\label{f+e}
f(A+\epsilon)=f_\epsilon(A).
\end{equation}

Since $-A$ is the generator of a $C_0$-semigroup $(T(t))_{t\ge0}$ on $X$,
and $\mathcal H_\psi$ contains the functions $e^{-t\cdot}$ for $t\ge0$, it is natural to expect
that
\[
 e^{-t\cdot}(A) = T(t), \qquad t \ge 0.
 \] 
This follows from the fact that if $t \ge 0$ then $e^{-t\cdot }$ is the limit in $\mathcal H_\psi$ of $\rho_1(t\cdot/n)^n, \, n\to\infty$, proved in Remark \ref{Tt} below. 

We will be concerned only with the exponential functions $(e^{-tz})_{t\ge0}$ and rational functions $r$ which are $\mathcal{A}(\psi)$-stable, and functions derived from these, applied in situations 
where $\psi\in(0,\pi/2)$, and $A$ is sectorial of angle $\theta\in [0,\psi)$. The operators $(e^{-tA})_{t\ge0}$ form a $C_0$-semigroup which extends to a bounded holomorphic semigroup on a sector. The operators $r(A)$ have the natural definition for rational functions.

\section{Estimates for rational approximations of the exponential}  \label{sect3}
\subsection{Stability estimates}

In this section we extend known results about the stability of certain classes of rational approximations to the exponential function with variable stepsizes.   In the next section we will apply these results to sectorial operators. For $R>0$, recall the notation 
$\mathbb D_R = \{z\in\C: |z|<R\}$.   
\begin{prop}\label{PderA}
Let $f$ be a holomorphic function on  $\mathbb D_R$ which is non-constant, $\theta \in (-\pi,\pi]$, and let $f_\theta(t):=|f(te^{i\theta})|, \,  0<t<R$. Then
there exist  $R_0\in (0,R)$ and a finite (possibly empty) set
$\Theta_f\subset (-\pi,\pi]$ such that for every $\theta\in (\pi,\pi]\setminus \Theta_f$,
\begin{equation}\label{P1D}
0<|f'(te^{i\theta})|\le C_\theta
|f'_\theta(t)|,\quad t\in (0,R_0),\quad
\end{equation}
for some $C_\theta >0$.
\end{prop}

\begin{proof}
By adjusting $R$ if necessary, we may assume that $f'$ has no zeroes in $\mathbb{D}_R\setminus \{0\}$.
From the Taylor series for $f$, there exist  $k\in\N$ and $b = |b|e^{i\varphi} \ne 0$ such that
\[
f(z)=f(0)+bz^k+\mbox{O}(z^{k+1}),\qquad z\to 0,\quad z\in \mathbb D_R,
\]
and
\begin{equation}\label{phiD}
|f'(te^{i\theta})|=k|b|t^{k-1}+\mbox{O}(t^k),\qquad t\to 0,\quad \theta\in (-\pi,\pi].
\end{equation}

If $f(0)=0$,
then
\[
f_\theta(t)=|b|t^k+\mbox{O}(t^{k+1}),\qquad
f'_\theta(t)=k|b|t^{k-1}+\mbox{O}(t^{k}),
\]
uniformly in $\theta$, and in view of (\ref{phiD}), the inequality (\ref{P1D}) holds
for some $R_0>0$ and every $\theta\in (-\pi,\pi]$.

If $f(0)\not=0$, we  may assume, without loss of generality, that
$f(0)=1$.   A simple calculation reveals that
\[
f_\theta^2(t)=1+2|b|t^k\cos(\varphi+k\theta)+\mbox{O}(t^{k+1}),\\
\]
uniformly in $\theta$.
Hence, for every $\theta\in (-\pi,\pi]$,
\begin{align*} \label{fth}
f_\theta(t) &= 1 + {\rm O} (t^k), \\
f_\theta(t)f'_\theta(t)
&=k|b|t^{k-1}\cos(\varphi+k\theta)+{\rm O}(t^k).
\end{align*}
Taking into account \eqref{phiD}, we infer that (\ref{P1D}) is satisfied
for some $R_0>0$ and
\[
\Theta_f :=\{\theta\in (-\pi,\pi]:\,\cos(\varphi+k\theta)=0\}.
\qedhere\]
\end{proof}

\begin{cor}\label{CPderA}
Let $r$ be a rational function which is non-constant, holomorphic at zero and bounded at infinity.
Then there exist a  finite (possibly empty) set
$\Theta_r\subset (-\pi,\pi]$
and $R_1>R_0>0$ such that
for every $\theta\in (-\pi,\pi]\setminus \Theta_r$
\begin{equation}\label{CP1D}
0<|r'(te^{i\theta})|\le c_\theta
|r'_\theta(t)|,\qquad t\in (0,R_0) \cup (R_1,\infty),
\end{equation}
for some $c_\theta >0$.
\end{cor}

\begin{proof}
Employing Proposition
\ref{PderA} we infer that there is $R_0>0$ such that the inequality in \eqref{CP1D} holds for $t \in (0,R_0)$ and all except finitely many $\theta$.
Passing then to $r^*(z):=r(1/z)$ and employing Proposition \ref{PderA} again
we obtain $R_1>0$ such that the inequality in \eqref{CP1D} also
holds for $t >R_1$ and all except finitely many $\theta$.
\end{proof}

A consequence of \eqref{CP1D} is that $r_\theta'$ has no zero in $(0,R_0)$ or $(R_1,\infty)$.  This has the consequence that if $g$ is a non-negative measurable function on an interval $I$ contained in $(0,R_0)$ then
\begin{equation} \label{signs}
\int_I |r_\theta'(t)| g(t) \,dt = \left| \int_I r_\theta'(t) g(t) \,dt \right|.
\end{equation}

\begin{exa}\label{CC1}
We illustrate Proposition \ref{PderA} by considering a Cayley transform.
Let  $z_0=\tau e^{i\varphi}\in \C_{+}$ and
\[
r(z)=\frac{z-z_0}{z+\overline{z}_0},\qquad z\in \C_{+}.
\]
We will show that
for every $\theta\in (-\pi/2,\pi/2)$,
\begin{equation}\label{relZB}
r_\theta'(t)<0,\quad t\in (0,\tau),\qquad \text{and}\qquad
r_\theta'(t)>0,\quad t\in (\tau,\infty),
\end{equation}
and for every $\delta\in (0,1)$ there exists  $c_\delta>0$ such that
\begin{equation}\label{relZ1B}
|r'(te^{i\theta})|\le c_\delta\frac{|r_\theta'(t)|}{\cos\theta},\quad
t\in (0,(1-\delta)\tau)\cup((1+\delta)\tau,\infty).
\end{equation}

Indeed, note that
\[
|r'(te^{i\theta})|=\frac{2\tau\cos\varphi}{t^2+2t\tau\cos(\theta+\varphi)+\tau^2},
\]
and
\[
r^2_\theta(t)=\frac{t^2-2t\tau\cos(\theta-\varphi)+\tau^2}{t^2+2t\tau\cos(\theta+\varphi)+\tau^2}.
\]
Differentiation gives
\[
r_\theta(t)r_\theta'(t)=\frac{2\tau (t^2-\tau^2)\cos\theta\cos\varphi}
{(t^2+2t\tau\cos(\theta+\varphi)+\tau^2)^2}  = \frac{(t^2-\tau^2) \cos\theta}{t^2+2t\tau\cos(\theta+\varphi)+\tau^2} |r'(te^{i\theta}|.
\]
Since $r_\theta(t)>0$ and $r_\theta(t)\le 1$, this shows that \eqref{relZB} holds.  Moreover,
\begin{align*}
|r'(te^{i\theta})| &= \frac{t^2+2t\tau\cos(\theta+\varphi)+\tau^2}{|t^2-\tau^2|\cos\theta}|r_\theta(t) r'_\theta(t)| \\
 &\le \frac{(t+\tau)^2}{|t^2-\tau^2| \cos\theta} |r'_\theta(t)|
= \frac{t+\tau}{|t-\tau|\cos\theta} |r'_\theta(t)|,
\end{align*}
so \eqref{relZ1B} holds with $c_\delta = 2\delta^{-1}-1$.
\end{exa}

In Example \ref{CC1},
$r\in H^\infty(\C_{+})$ and (\ref{CP1D}) holds for all $\theta \in (-\pi/2,\pi/2)$.
In the general case, when $r \in H^\infty (\Sigma_\psi)$ and $\psi\in (0,\pi/2]$,
there may be exceptional values of $\theta$ for which
(\ref{CP1D}) does not hold as the following two examples show.

\begin{exa}
Fix
$\varphi\in (0,\pi/2)$ and consider the rational function
\[
r(z):=\frac{z-e^{-i\varphi}}{z+e^{-i\varphi}}.
\]
Note that $r\in H^\infty(\C_{+})$.  If
$
z=i e^{-i\varphi}t=t e^{i(\pi/2-\varphi)}, t>0,
$
then $z \in \C_+$ and
\[
|r(ie^{-i\varphi}t)|=\left|\frac{it-1}{it+1}\right|=1,\qquad t>0.
\]
Hence
\[
\frac{d}{dt}|r(ie^{-i\varphi}t)|\equiv 0,\quad t>0,
\]
and so (\ref{CP1D}) does not hold for $\theta = \pi/2-\varphi$.
\end{exa}

\begin{exa}
For a fixed $\varphi\in (-\pi/2,\pi/2)$ define (a ``shifted Cayley
transform")
\[
r(z):=\frac{z+1-e^{-i\varphi}}{z+1+e^{i\varphi}}
\]
and note that
$\|r\|_{{H}^\infty(\C_{+})}=1$ and $r(\infty)=1$.   Moreover, for every $t>0$,
\[
|r(t)|^2
=\frac{(t+1)^2-2(t+1)\cos\varphi+1}
{(t+1)^2+2(t+1)\cos\varphi+1},
\]
By simple calculations,
\[
\left.
\frac{d}{dt}|r(t)|\right|_{t=0}=0, \qquad |r'(0)|=\frac{2}{|1+e^{i\varphi}|^2}\not=0.
\]
This implies that \eqref{CP1D} does not hold for $\theta=0$ and small $t>0$.
\end{exa}

Now we turn our attention towards rational approximations schemes involving
 products of a single function over variable stepsizes.
Let $r$ be an $\mathcal{A}(\psi)$-stable rational function,
and
let $\theta \in (0,\psi)$.
Let  $\mathcal{K}_n = (k_j)_{j=1}^n$ be a sequence in $(0,\infty)$ of length $n$,
$
K_0=\min_j\,\{k_j\}$ and $K_1=\max_j\,\{k_j\}$.
Let
\begin{equation}\label{defp}
P_{\mathcal{K}_n,r} (z):=\prod_{j=1}^n r(k_jz),   \qquad z \in \Sigma_\psi.
\end{equation}
By \eqref{derivative} and \eqref{normatt},
we have $P_{{\mathcal K}_n, r} \in \mathcal H_{\theta, 0}$ and
\begin{equation}\label{normh}
\|P_{{\mathcal K}_n, r}\|_{\mathcal H_{\theta, 0}}=\|P_{{\mathcal K}_n, r}'\|_{L^1(\partial \Sigma_\theta)}.
\end{equation}
We will give two estimates for the right-hand side of \eqref{normh}, firstly an estimate which is valid for all $\mathcal{A}(\psi)$-stable rational functions, but the estimate depends on $\mathcal{K}_n$ as well as $r$ and $\theta$, and secondly an estimate which is independent of $\mathcal{K}_n$ provided that $|r(\infty)|<1$.

In the statements and proofs of these two theorems, the notation $C$ denotes a constant which may depend on $r$, $\psi$ and $\theta$, but does not depend on $\mathcal{K}_n$.   The value of $C$ may change from place to place.

\begin{thm}\label{T1Hol}
Let $r$ be an $\mathcal{A}(\psi)$-stable rational function, where $\psi \in (0,\pi/2]$, and let $\theta\in(0,\psi)$.
There is a constant $C$ such that
\begin{equation}\label{boundd}
\|P_{\mathcal{K}_n, r}\|_{\mathcal H_{\theta, 0}} \le C \left(1 + \log\left(\frac{K_1}{K_0}\right)\right)
\end{equation}
 for all finite sequences $\mathcal{K}_n$ in $(0,\infty)$.
\end{thm}

\begin{proof}
Let $R_0$, $R_1$ and $\Theta_r$ be as in Corollary \ref{CPderA}, and
\[
a:=\frac{R_0}{K_1},\qquad b:=\frac{R_1}{K_0}.
\]
By the maximum principle, it suffices to consider $\theta \in (0,\psi)$, sufficiently close to $\psi$ that (\ref{CP1D}) holds for $\pm\theta$.

Let $\mathcal{K}_n$ and $P_{\mathcal{K}_{n},r}$ be as above.
 Since $\|P_{\mathcal{K}_{n},r}\|_{H^\infty(\Sigma_\psi)} \le 1$,
 there is a constant $C$ (depending only on $\theta$ and $\psi$)  such that
\begin{equation}\label{rR}
|P_{{\mathcal K}_n, r}'(z)|\le \frac{C}{|z|},\qquad z\in \overline{\Sigma}_\theta, \,z\ne0.
\end{equation}
Using (\ref{CP1D}),  (\ref{rR}) and (\ref{signs}), we have
\begin{align*}
\lefteqn{\hskip-15pt \int_0^\infty \left|P_{{\mathcal K}_n, r}'(te^{i\theta})\right|\,dt}\\
&\le
\int_0^a \sum_{j=1}^n k_j|r'(k_jt e^{i\theta})|\prod_{l=1,\;l\not=j}^n
|r(k_lte^{i\theta})|\,dt\\
&\hskip15pt \null+\int_a^b
\left|P_{{\mathcal K}_n, r}'(te^{i\theta})\right|\,dt
+\int_b^\infty \sum_{j=1}^n k_j|r'(k_jt e^{i\theta})|\prod_{l=1,\;l\not=j}^n
|r(k_lte^{i\theta})|\,dt\\
&\le C
\Bigg|\int_0^a \sum_{j=1}^n \frac{d}{dt}|r(k_jt e^{i\theta})|\prod_{l=1,\;l\not=j}^n
|r(k_lte^{i\theta})|\,dt\Bigg|\\
&\hskip15pt\null +C \int_a^b \frac{dt}{t}
+C \Bigg|
\int_b^\infty \sum_{j=1}^n \frac{d}{dt}|r(k_jt e^{i\theta})|\prod_{l=1,\;l\not=j}^n
|r(k_lte^{i\theta})|\,dt\Bigg|\\
&=C \Bigg|\int_0^a \frac{d}{dt}\Big|\prod_{j=1}^n r(k_jte^{i\theta})\Big|
\,dt\Bigg|
+C \log\left(\frac{b}{a}\right)
+C \Bigg|\int_b^\infty \frac{d}{dt}\Big|\prod_{j=1}^n r(k_jte^{i\theta})\Big|
\,dt\Bigg| \\
&\le C+C \log\left(\frac{R_1K_1}{R_0K_0}\right)  \le  C\left( 1 + \log\left(\frac{K_1}{K_0}\right) \right).
\end{align*}
The equality sign is valid by the sign argument in (\ref{signs}).
The same argument also works for $-\theta$, and hence \eqref{boundd} follows.
\end{proof}

\begin{thm}\label{infty}
Let $r$ be an $\mathcal{A}(\psi)$-stable rational function
with $|r(\infty)|<1$,
and let $\theta \in (0,\psi)$.
There exists $C$ such that
\begin{equation}\label{bound}
\| P_{\mathcal K_n, r} \|_{{\mathcal H}_{\theta, 0}} \le C
\end{equation}
for every finite sequence $\mathcal{K}_n \subseteq (0,\infty)$.
\end{thm}

\begin{proof}
Without loss of generality, we may assume that $R_0=1$;  otherwise consider $r(zR_0)$ in place of $r(z)$, noting that this change of variable does not change $\|P_{\mathcal{K}_n,r}\|_{\H_{\theta,0}}$.
Thus
\begin{equation}\label{ab1}
0<|r'(t e^{i\theta})|\le C |r'_\theta(t)|,\qquad
t\in (0,1).
\end{equation}
By the maximum principle there exists $\kappa_r\in (0,1)$ such that
\begin{equation}\label{InfB}
|r(z)|\le \kappa_r,\qquad z\in \Sigma_\theta,\; |z|\ge 1.
\end{equation}

Let $\mathcal{K}_n$ be a sequence of length $n$ in  $(0,\infty)$.   
We have shown in Theorem \ref{T1Hol} that $P_{\mathcal K_n, r}\in{\mathcal H}_{\theta, 0}$.   
To obtain \eqref{bound}, we will use (\ref{InfB}) to improve the estimate of the $L^1$-norm of $P_{\mathcal K_n, r}$
on the boundary of $\Sigma_\theta$, where $\theta\in(0,\psi)$ and $(\ref{CP1D})$ holds for $\pm\theta$.

Let $\mathcal{K}_n = (k_j)_{j=1}^n$, where $0 < k_1 \le k_2 \le \dots \le k_n$ without loss of generality, and let
\begin{equation}\label{LL}
t_0=0, \;\quad t_j=\frac{1}{k_{n-j+1}},\; j=1,\dots,n, \quad\; t_{n+1}=\infty.
\end{equation}
Let
\[
S_{r,m}(\mathcal{K}_n;\theta):=\int_{t_{m-1}}^{t_m} \big|
P'_{\mathcal{K}_n,r}(t e^{i\theta})\big|\,dt, \qquad m=1, \dots, n+1.
\]
Then
\begin{equation}\label{SST}
\int_0^\infty \big|P'_{\mathcal K_n, r}(te^{i\theta})\big|\,dt
= \sum_{m=1}^{n+1} S_{r,m}(\mathcal{K}_n;\theta). 
\end{equation}
Taking into account that
\[
k_jt\le 1,\quad j=1,\dots,n,\quad t\in (0,1/k_n),
\]
and applying \eqref{LL} and \eqref{signs},
we have
\begin{align*}
 S_{r,1}(\mathcal{K}_n;\theta)
&\le \sum_{j=1}^n k_j\int_0^{1/k_n}
|r'(k_je^{i\theta}t)|\prod_{l=1,\;l\not=j}^n |r_\theta(k_lt)|\,dt\\
&\le C\sum_{j=1}^n k_j\int_0^{1/k_n}
\left |\frac{d}{dt} r_\theta(k_jt)\right|\prod_{l=1,\;l\not=j}^n |r_\theta(k_lt)|\,dt\\
&= C\left|\int_0^{1/k_n}
\frac{d}{dt}|P_{\mathcal{K}_n,r}(te^{i\theta})|\,dt\right|\le C.
\end{align*}
For $2 \le m \le n$, we have
\begin{align*}
k_j t&\le 1,\quad j=1,\dots,n-m+1,\quad
t\le \frac{1}{k_{n-m+1}},\\
k_l t&\ge 1,\quad l=n-m+2,\dots,n,\quad  t\ge \frac{1}{k_{n-m+2}}.
\end{align*}
Using \eqref{ab1}, \eqref{derivative}, \eqref{InfB} and \eqref{signs}, we  obtain
\begin{align*}
 S_{r,m}(\mathcal{K}_n;\theta)
&\le\sum_{j=1}^n k_j\int_{1/k_{n-m+2}}^{1/k_{n-m+1}}
|r'(k_je^{i\theta}t)|\prod_{l=1,\;l\not=j}^n |r_\theta(k_lt)|\,dt\\
&\le C \kappa_r^{m-1} \sum_{j=1}^{n-m+1} k_j\int_{1/k_{n-m+2}}^{1/k_{n-m+1}}
\left|\frac{d}{dt}r_\theta(k_jt)\right|\prod_{l=1,\;l\not=j}^{n-m+1} |r_\theta(k_lt)|\,dt\\
&\hskip15pt\null + \sum_{j=n-m+2}^{n} k_j\int_{1/k_{n-m+2}}^{1/k_{n-m+1}}
|r'(k_{j}e^{i\theta}t)|\prod_{l=1,\;l\not=j}^n |r_\theta(k_l t)|\,dt\\
&= C \kappa_r^{m-1}\left|\int_{1/k_{n-m+2}}^{1/k_{n-m+1}}
\frac{d}{dt}
|P_{n-m+2}(te^{i\theta})|\,dt\right|\\
&\hskip12pt\null + c_r\kappa_r^{m-2} \sum_{j=n-m+2}^{n} \int_{1/k_{n-m+2}}^{1/k_{n-m+1}}
\frac{k_{j}}{(1+k_{j}t)^2}\,dt\\
&\le C\kappa_r^{m-1}+c_r(m-1)\kappa_r^{m-2}.
\end{align*}
For $m=n+1$, we have
\begin{align*}
 S_{r,n+1}(\mathcal{K}_n;\theta)
&\le \sum_{j=1}^n k_j\int_{1/k_1}^\infty
|r'(k_je^{i\theta}t)|\prod_{l=1,\;l\not=j}^n |r_\theta(k_lt)|\,dt\\
&\le c_r\kappa_r^{n-1}\sum_{j=1}^n \int_{1/k_1}^\infty
\frac{k_j}{(1+k_jt)^2}\,dt \le c_r n\kappa_r^{n-1}.
\end{align*}
Thus,
\begin{align*}
\int_0^\infty \big|P'_{\mathcal K_n, r}(te^{i\theta})\big|\,dt
&\le C
+\sum_{m=2}^n (C\kappa_r^{m-1}+c_r(m-1)\kappa_r^{m-2})
+ c_r n\kappa_r^{n-1}\\
&=C \sum_{m=1}^n \kappa_r^{m-1}+c_r \sum_{m=1}^n m\kappa_r^{m-1}\le \frac{C}{1-\kappa_r}+
\frac{c_r}{(1-\kappa_r)^2}.
\end{align*}
and completely similarly
\begin{align*}
\int_0^\infty \big|P'_{\mathcal K_n, r}(te^{-i\theta})\big|\,dt
\le \frac{C}{1-\kappa_r}+
\frac{c_r}{(1-\kappa_r)^2}.
\end{align*}
Now  \eqref{bound} is immediate.
\end{proof}

\subsection{Approximation rate estimates} \label{sect4}

In this section, we prove several auxiliary lemmas about $\mathcal{A}(\psi,m)$-stable rational approximations of the exponential function, which lead to the improved version of Theorem \ref{basic}  by means of
the $\mathcal H$-calculus. Recalling the definition of $A(\psi, m)$-stable rational approximations from the Introduction, we start with a simple example of an $\mathcal{A}(\pi/6,2)$-stable function which approximates $e^{-z}$ with $q=1$ and $|r(\infty)|=1$, illustrating this new sharper notion of stability.

\begin{exa}
Let
\[
r(z) = \frac{1-4z^3}{1+z+4z^3}, \quad z\in\C_+.
\]
Then $r(\infty)=-1$, and
\[
r(z)+1 = \frac{2+z}{1+z+4z^3} = \frac{1}{4z^2} + {\rm O}(|z|^{-3}), \qquad z\to\infty.
\]
Moreover,
\[
r(0)=1, \qquad r'(0) = -1, \qquad r''(0)=2,
\]
so $r$ approximates $e^{-z}$ with exact order $q=1$.   The poles of $r$ are at $-1/2$ and at
\[
\frac{1 \pm i\sqrt{7}}{4} = \frac{e^{\pm i\varphi}}{\sqrt{2}}, \quad \varphi = \arctan\sqrt{7} \in (\pi/3,\pi/2).
\]
Thus $r$ is holomorphic on $\Sigma_{\pi/6}$. If $t>0$ and $z = te^{i\pi/6} = (\sqrt{3}+i)t/2$, then $z^3 = it^3$ and
\begin{align*}
|r(te^{\pm i\pi/6})|^2 &= \frac{16t^6+1}{(4t^3+t/2)^2 + (\sqrt{3}t/2+1)^2} \\
&= \frac{16t^6+4}{16t^6 + 4t^4+t^2 +\sqrt{3}t + 4} < 1.
\end{align*}
Thus $r$ is $\mathcal{A}(\pi/6,2)$-stable.
\end{exa}

For the study of approximation rates for rational approximations 
$r(t\cdot/n)^n$ to the exponential, 
it is crucial  to have sharp estimates for the size
of the rational function $r$ near infinity. The following lemma
provides such an estimate in terms the Taylor expansion of $r$
at infinity.
\begin{lemma}\label{rin}
Let $r$ be an $\mathcal A(\psi, m)$-stable rational function with $|r(\infty)|=1$.
Then for every $\theta\in (0,\psi)$ there exist $b_1,b_2 >0$ and $R \ge 1$,
such that
\begin{equation}\label{double}
e^{-b_2/|z|^m} \le |r(z)|\le e^{-b_1/|z|^m}
\end{equation}
for all $z\in \Sigma_\theta$ with $|z|\ge R$.
\end{lemma}

\begin{proof}
Without loss of generality, we may assume  that $r(\infty)=1$.
If $a=|a|e^{i\beta}$, $\beta \in (-\pi, \pi]$, then it follows from
$\mathcal{A}(\psi,m)$-stability that
\[
|r(te^{i\varphi})|^2
=1-2|a|t^{-m}\cos(\beta-m\varphi)+\mbox{O}(t^{-m-1}),\qquad t\to\infty, \quad |\varphi|\le \psi.
\]
It follows that
\[
\cos(\beta\pm m\varphi)\ge 0,\qquad |\varphi|\le \psi,
\]
and then
\[
|\beta \pm m\psi |\le \pi/2,
\]
and finally
\[
|\beta|+m\psi= \max\{|\beta-m\psi|, |\beta+m\psi|\}
\le \pi/2.
\]
Let $z=te^{i\varphi}$, with $t > 0$, $|\varphi|\le \theta$, and let
\[
\omega:=\max\,\{|\beta-m\theta|, |\beta+m\theta|\}<\pi/2.
\]
We infer from \eqref{order} that there exists $c>0$ such that
\begin{equation}\label{upper0}
|r(te^{i\varphi})|\le  \left|1-\frac{|a|e^{i\beta}e^{-im\varphi}}{t^m}\right|
+\frac{c}{t^{m+1}}
\le
\left|1-\frac{|a|e^{i\omega}}{t^m}\right|
+\frac{c}{t^{m+1}}.
\end{equation}
If
\begin{equation}\label{ineq1}
 t\ge (|a|/\cos\omega)^{1/m},
\end{equation}
we may apply an elementary inequality
\[
|1-z|\le 1-\frac{{\rm Re}\,z}{2}\qquad \mbox{if} \qquad |z|^2\le {\rm Re}\,z,
\]
and conclude that
\begin{equation} \label{upper}
|r(te^{i\varphi}| \le 1-\frac{|a|\cos\omega}{2t^m}+\frac{c}{t^{m+1}}.
\end{equation}
If in addition to \eqref{ineq1},
\begin{equation}\label{ineq2}
t\ge\left( \frac{4c}{|a|\cos\omega}\right)^{1/(m+1)},
\end{equation}
then it follows from \eqref{upper} and the inequality $1- \tau \le e^{-\tau}$ that
\begin{equation}\label{upper1}
|r(te^{i\varphi})|\le  1-\frac{|a|\cos\omega}{4t^m}
\le e^{-|a|\cos\omega/(4t^m)}.
\end{equation}
Thus the second inequality in \eqref{double} holds with $b_1=(|a|\cos\omega)/4$ and
\[
 t \ge R_1:=\max\,\left\{\left(\frac{|a|}{\cos\omega}\right)^{1/m},
\left(\frac{4c}{|a|\cos\omega}\right)^{1/(m+1)}\right\}.
\]

For the first inequality, there exists $c'>0$ such that
\begin{equation}\label{ineqr}
|r(te^{i\varphi})|\ge 1-\frac{|a|}{t^m}-\frac{c'}{t^{m+1}}, \qquad t>0.
\end{equation}
Let
\begin{equation}\label{r2}
R_2:=\max\,\{c'/|a|,(8|a|)^{1/m}\}.
\end{equation}
If $t \ge R_2$, then the elementary bound $1-\tau/2 \ge e^{-\tau}$, $\tau\in [0,1]$, shows that
\[
|r(te^{i\varphi})|\ge 1-\frac{|a|}{t^m}-\frac{c'}{t^{m+1}}
\ge 1-\frac{2|a|}{t^m} \ge e^{-4 |a|/t^m}.
\]
Hence the first inequality in \eqref{double} is true with with $b_2=4|a|$ and $t \ge R_2$ as in (\ref{r2}).
So \eqref{double} holds with $R=\max(R_1, R_2, 1)$.
\end{proof}

Given an $\mathcal{A}(\psi)$-stable rational function $r$ and  $n \in \N$,
define
\[
r_n(z):=r^n(z/n).
\]
For $\varphi \in [-\psi, \psi]$, $n \in \mathbb N$, $s\ge0$, $\epsilon \ge 0$ and $R>0$, let
\begin{align} \label{QRd}
& Q^\epsilon_{R}(\varphi, n, s):=\int_{Rn}^{\infty}
\left| \frac{d}{dt}\left(\frac{r_n(te^{i\varphi})}{(te^{i\varphi}+\epsilon)^s}\right)\right|\,
dt \\
&= \frac{1}{n^s} \int_R^\infty \left|n (u e^{i\varphi}+\epsilon)r'(ue^{i\varphi})-s r(u e^{i\varphi}) \right| \,|r^{n-1}(u e^{i\varphi})|\, \frac{du}{|ue^{i\varphi}+\epsilon|^{s+1}},\notag
\end{align}
by a change of variable $t = ne^{i\varphi}u$.

The next rather technical lemma is basic for obtaining our improvements of approximation rates
for rational approximations.
\begin{lemma}\label{Ac1}
Let $r$ be an $\mathcal A(\psi, m)$-stable rational function such that $|r(\infty)|=1$,
and let $\theta \in (0,\psi)$ and $T>0$.
Then there exist $R\ge1$ and $C_1=C_1(\theta, T)>0$ and $C_2=C_2(\theta, \epsilon, T)>0$
such that for all $\varphi \in [-\theta, \theta]$, $s \in [0, T]$, and $n \in \N$,
\begin{equation}\label{QQA}
Q^0_R(\varphi, n, s)\le \frac{C_1}{n^{\msm}},
\end{equation}
and, for every $\epsilon\ge 0$,
\begin{equation}\label{QQAEps}
Q^\epsilon_R(\varphi, n, s) \ge \frac{C_{2}}{n^{\msm}}.
\end{equation}
\end{lemma}

\begin{proof}
Let $b_1$, $b_2$ and $R\ge 1$ be as in Lemma \ref{rin}.   Let $\varphi \in [-\theta,\theta]$, and $s\ge0$.  Note that $|r(te^{i\varphi})|\le 1$ and, from (\ref{order}), that
\begin{equation}\label{use}
|r'(z)|\le \frac{c_\theta}{|z|^{m+1}},\qquad z\in \Sigma_\theta.
\end{equation}
With the changes of variables $\tau= t^{-m}$ and $v=(n-1)\tau$ for $n\ge2$ in \eqref{QRd},
we can estimate as follows for $\epsilon=0$:
\begin{align*}\label{Al}
Q^0_{R}(\varphi, n, s)
&\le \frac{1}{n^s}\int_{R}^\infty \,
(c_\theta nt^{-m}+s)|r^{n-1}(t e^{i\varphi})|
\,\frac{dt}{t^{s+1}}\notag \\
&\le \frac{1}{mn^s}\int_0^{1/R^m}
(c_\theta n\tau+s)\,e^{-b_1(n-1)\tau} \tau^{s/m-1}\,d\tau\\
&\le  \frac{C(\theta, m,b_1)}{n^s}\int_0^\infty
(n\tau+s) \tau^{s/m-1}\,e^{-b_1(n-1)\tau}\,d\tau\\
&\le \frac{C(\theta, m,b_1)}{(n-1)^{s/m}n^s} \int_0^\infty
(2v+s) v^{s/m-1}\,e^{-b_1v}\,dv\\
&= \frac{C(\theta, m,b_1)(2+mb_1)}{b_1}  \frac{\Gamma(1+s/m)}{b_1^{s/m}{(n-1)^{s/m}n^s}},
\end{align*}
where $\Gamma$ stands for the gamma function. For $T>0$, this estimate is bounded by an expression of the form on the right-hand side of \eqref{QQA}, where $C_1$ is independent of $s \in [0,T]$ and $n\ge2$.

For $n=1$, we have
\[
Q^0_R(\varphi,1,s) \le \frac{1}{m} \int_0^{1/R^m} (c_\theta\tau+s) \tau^{s/m-1}\,d\tau,
\]
which is easily seen to be finite and bounded for $s\in[0,T]$, so there is an upper bound as in \eqref{QQA}, valid for all $\varphi\in[-\theta,\theta]$, $s \in [0,T]$ and $n\in \N$.

Now we will justify the lower bound in \eqref{QQA}. Let $\epsilon \ge 0$ be fixed. There exists $\tilde c_\theta>0$ such that
\begin{equation}\label{QQa}
|r'(z)|\ge \frac{\tilde{c}_\theta}{|z|^{m+1}},\qquad z\in \Sigma_\theta,\quad
|z|\ge R.
\end{equation}
Let $T>0$, and
\[
R_T=
\frac{1}{2}\left(\frac{\tilde{c}_\theta}{2T}\right)^{1/m}.
\]
Let $s \in [0,T]$.
Let $n \ge (R/R_T)^m$ and $z = t e^{i\varphi}$, where $|\varphi| <  \theta$ and $R_T n^{1/m} \le t \le 2R_T n^{1/m}$.   Then
\[
n|z+\epsilon|\ge n |z| \, |r'(z)| \ge 2T\ge 2s.
\]
Using  \eqref{QQa}, \eqref{double} \and  the inequalities
\[
e^{-b_2(n-1)/t^m}\ge e^{-b_2/R_T^m},\qquad t\ge R_T n^{1/m},
\]
and
\[
|te^{i\varphi}+\epsilon|\le (1+\epsilon)t,\qquad t\ge 1,
\]
 we infer that
\begin{align*}
(1+\epsilon)^{s+1}Q^\epsilon_{R}(\varphi, n, s)
&\ge \frac{1}{n^s}\int_{R_T n^{1/m}}^{2R_T n^{1/m}}\,
(nt |r'(te^{i\varphi})|-s)
|r^{n-1}(te^{i\varphi})|
\,\frac{dt}{t^{s+1}}\\
&\ge \frac{n}{2n^s}\int_{R_T n^{1/m}}^{2R_T n^{1/m}}\,|r'(t e^{i\varphi})|
|r^{n-1}(te^{i\varphi})|
\,\frac{dt}{t^{s}}\\
&\ge \frac{n}{2n^s}
\int_{R_T n^{1/m}}^{2R_T n^{1/m}} e^{-b_2(n-1)/t^m}
\,\frac{dt}{t^{s+m+1}}\\
&\ge \frac{n}{2n^s}
\int_{R_T n^{1/m}}^{2R_T n^{1/m}} e^{-b_2/R_T^m}
\,\frac{dt}{t^{s+m+1}}\\
&= \frac{\left(1 - 2^{-(s+m)} \right) e^{-b_2 / R_T^m}}{2(s+m)  R_T^{s+m}} \,n^{-\msm} .
\end{align*}
By taking the minimum value of the term for $s \in [0,T]$, we obtain a lower bound of the form \eqref{QQAEps} for all but a finite set of $n$.   For each of the remaining values of $n$, the function
\[
(\varphi,s) \mapsto n^{\msm} Q^\epsilon_R(\varphi,n,s)
\]
is continuous on $[-\theta,\theta] \times [0,T]$, so it attains a positive minimum value.
  Thus a lower bound as in \eqref{QQA} can be obtained for all $\varphi\in[-\theta,\theta]$, $s \in [0,T]$, and $n\in\N$.
 \end{proof}

For an $\mathcal{A}(\psi)$-stable rational function $r$, $n\in\N$ and $s\ge 0$, we define
\[
\Delta_{n, s}(z)=\frac{e^{-z}-r_n(z)}{z^s},
\qquad z\in \Sigma_\psi,\quad n\in \N.
\]
A somewhat tedious proof of the next auxiliary estimate is postponed to the Appendix.  
\begin{lemma}\label{asymptotic}
Let $r$ be an $A(\psi)$-stable rational approximation of order $q$ to
the exponential, and let $\theta \in (0, \psi)$ and $R>0$. Then there exist $C > 0$ and
$\a>0$ depending on $r, \theta$ and $R$ such that
\begin{equation}\label{Main111}
\left|\Delta_{n, s}'(z) \right|\le
C\left(q+1-s+|z|\right) \frac{|z|^{q-s}}{n^q}e^{-\a|z|}
\end{equation}
for all $n \in \mathbb N$, $z \in \Sigma_\theta$, $|z| < Rn$, and $s \in [0, q+1]$.
\end{lemma}

We also need the following estimate, under the assumptions of Lemma \ref{asymptotic} and the additional assumption that $R \ge 1$:
\begin{align} \label{D-r}
\int_{Rn}^{\infty} \left|\frac{d}{dt}\left(\frac{e^{-te^{i\varphi}}}{(te^{i\varphi}+\epsilon)^s}\right)\right| \,
dt
&\le \int_{Rn}^\infty \left( \frac{s}{t^{s+1}} + \frac{1}{t^s} \right) e^{-t\cos\varphi} \, dt\\
&\le \frac{q+2}{\cos\varphi} e^{-Rn\cos\varphi}.\notag
\end{align}

\begin{rem}\label{rem_eps}
It should be noted that Lemmas \ref{rin} and \ref{Ac1} as well as the estimate in \eqref{D-r} apply in exactly the same form for the shifts of $r$, i.e., the functions $r_\epsilon(z) = r(z+\epsilon)$, with possibly different values for $C=C(\epsilon)$ depending on $\epsilon$.
\end{rem}

Lemmas \ref{Ac1} and \ref{asymptotic}  lead to the following function-theoretical result
on rational approximations, which after applying the $\mathcal H$-calculus,
will transform into a statement on rational approximation rates
for holomorphic semigroups.
\begin{thm}\label{T1PH}
Let $r$ be an $\mathcal{A}(\psi,m)$-stable rational
approximation of order $q$ to
the exponential, and let $\theta\in (0,\psi)$.
Then the following hold.
\begin{enumerate}[{\rm(i)}]
\item  If $|r(\infty)|=1$, then there exists $C=C(\theta)$ such that
\begin{equation}\label{ThAPH}
\|\Delta_{n,s}\|_{{\mathcal H}_{\theta,0}}\le \frac{C}{n^{\delta_s}}
\end{equation}
for all $s\in[0,q+1]$ and $n\in\N$, where $\delta_s:=\min\,\{\msm,q\}$.
\item If $|r(\infty)|<1$, then there exists $C=C(\theta)$ such that
\begin{equation}\label{gammaR}
\|\Delta_{n,s}\|_{{\mathcal H}_{\theta,0}}\le \frac{C}{n^q}
\end{equation}
for all $s \in [0,q+1]$ and $n\in\N$.
\end{enumerate}
\end{thm}

\begin{proof}
We start with a proof of (i). Let $s \in [0,q+1]$ be fixed.
Let $R\ge 1$ be given by Lemma \ref{Ac1} with $T=q+1$, and fix $C$ and $\a$ given
by Lemma \ref{asymptotic} for that value of $R$.
Using \eqref{Main111}, \eqref{D-r} and Lemma \ref{Ac1},
we obtain
that
\begin{align}\label{deltans}
\|\Delta_{n,s}\|_{{\mathcal H}_{\theta,0}}
&\le \frac{2C}{n^q}\int_0^{Rn}
(q+1-s+t)t^{q-s}e^{-\a t}\,dt\\
&\hskip15pt\null +2  \int_{Rn}^\infty \left| \Delta'_{n,s}(te^{i\theta}) - r'_{n,s}(te^{i\theta}) \right| \,dt
+ 2 \int_{Rn}^\infty \,|r_{n,s}'(te^{i\theta})|\,dt  \notag \\
&\le\frac{2C}{n^q}\int_0^\infty (q+1-s+t) t^{q-s}e^{-\a t}\,dt\notag \\
&\hskip15pt\null+ \frac{2(q+2)}{\cos\theta} e^{-Rn\cos\theta}
+2Q^0_{R}(\theta, n, s)\notag \\
&\le C_1
\left(\frac{1}{n^q}+\frac{1}{n^{\msm}}\right)
\le \frac{C_2}{n^{\delta_s}},\notag
\end{align}
for appropriate constants $C_1$ and $C_2$, so (i) is proved.

Now assume that $|r(\infty)|<1$.  We use the same approach as for (i), taking $R=1$ in Lemma \ref{asymptotic}.   We follow the argument for (i) but Lemma \ref{Ac1} is inapplicable, and we need to estimate $Q_1^0(\varphi,n,s)$ in a different way.   By the assumptions and the maximum principle, we observe that, for $z \in \Sigma_\theta$ with $|z|>1$,
\begin{equation}\label{assump1}
|r(z)| \le  \kappa,\qquad |r'(z)|\le \frac{c_0}{|z|^2},
\end{equation}
for some $\kappa \in (|r(\infty)|,1)$, $c_0=c_0(\theta)$.
Setting $z=te^{i\varphi}$ and using \eqref{assump1} in \eqref{QRd}, we infer that there exists $C=C(\theta)$ such that for every $\varphi\in (0,\theta)$,
\[
Q_1^0(\varphi,n,s)
\le \frac{\kappa^{n-1}}{n^s} \int_{n}^\infty
\left(\frac{c_0 n^2}{t^{s+2}}+\frac{s}{t^{s+1}}\right)\,dt
\le \frac{C}{n^q}.
\]
Applying this estimate to replace Lemma \ref{Ac1} and Lemma \ref{asymptotic} with $R=1$,
and arguing as in \eqref{deltans},
we obtain  \eqref{gammaR}.
\end{proof}
\section{Stability of operator approximations} \label{sect_new}

Now we are ready to formulate our stability estimates for rational approximations
of (sectorially) bounded holomorphic semigroups. 
Since the negative  generators of such semigroups coincide with sectorial operators of angle less
than $\pi/2,$ we employ the language of sectorial operators,
and deduce our results as direct corollaries
of the $\mathcal H$-calculus, using  \eqref{estimate} along with Theorems \ref{T1Hol} 
and \ref{infty}. Firstly, we obtain
the following general statement
on stability of variable stepsize rational approximations first proved in \cite{Bakaev96}
and containing a well-known, constant-step stability result from \cite{Crou}, \cite{Larsson} (see also \cite{Thomee_b} and \cite{Haase_b}).
 Recall that for a finite sequence $\mathcal{K}_n := (k_j)_{j=1}^n \subset (0,\infty)$, we let $K_0 = \min \{k_j\}$ and $K_1 = \max\{k_j\}$.
  The proof is an immediate consequence of Theorem \ref{T1Hol} and \eqref{estimate}. 

\begin{thm}\label{HOp}
Let $r$ be an $\mathcal{A}(\psi)$-stable rational function, where $\psi \in (0,\pi/2]$, and let $\theta\in(0,\psi)$.
Then there exists $C = C(\theta)$ such that,
for every operator $A \in \Sect(\varphi)$ for some $\varphi\in(0,\theta)$ and every finite sequence $\mathcal{K}_n \subset (0,\infty)$,
\[
\bigg\|\prod_{j=1}^n r(k_jA)\bigg\|
\le  C M_\theta(A)\left(1 +
\log\left(\frac{K_1}{K_0}\right)\right).
\]
In particular, for any $n\in \N$,
\[
\left\|r^n(A/n)\right\|\le CM_\theta(A).
\]
\end{thm}

Secondly, our methods allow us to  strengthen substantially
a well-known result on variable stepsize approximation due to Palencia \cite{Palencia} and Bakaev \cite{Bakaev} (reproved in \cite{Bakaev98}) and to obtain the next statement.
It is an immediate consequence of Theorem \ref{infty} and \eqref{estimate}.

\begin{thm}\label{infty_op}
Let $r$ be an $\mathcal{A}(\psi)$-stable rational function with $|r(\infty)|<1$, where $\psi \in (0,\pi/2]$, and let $\theta\in(0,\psi)$.
Then there exists $C = C(\theta)$ such that,
for every operator $A \in \Sect(\varphi)$ for some $\varphi\in(0,\theta)$ and every finite sequence 
$\mathcal{K}_n \subset (0,\infty)$,
\[
\bigg\|\prod_{j=1}^n r(k_jA)\bigg\|
\le CM_\theta(A).
\]
\end{thm}

The assumptions in \cite{Bakaev} and \cite{Palencia} (and in \cite{Bakaev98}) included not only that $|r(\infty)|<1$, but also that $r$ approximates $e^{-z}$ with order $1$.  This assumption is indispensable in their proofs which are rather technical. For results related to Theorem \ref{infty_op} see also \cite{BakaevO}
where the first order of approximation also appears to be crucial.

\section{Rates for operator approximations}  \label{sect6}

The $\mathcal H$-calculus also leads to sharp estimates of approximation rates,
in particular to a refined version of Theorem \ref{basic}.
It might appear to be a mere improvement of Theorem \ref{basic},
taking into account fine structure of $r$.
However Theorem \ref{T1PHA}(i) gives a substantially better estimate than Theorem \ref{basic}(i) 
when $s\in (0,q)$
(and it is essentially vacuous when $s=0$).
Theorem \ref{T1PHA}(ii) for $s=0$ gives the result in Theorem \ref{basic}(ii),
and extends it slightly for smooth vectors $x$.
Moreover, the bounds given by Theorem \ref{T1PHA} are optimal as we prove in Section \ref{sect8}.
The proofs are of a common form, based on the $\H_\psi$-calculus
and its compatibility with the extended holomorphic calculus.

\begin{thm}\label{T1PHA}
Let $r$ be an $\mathcal{A}(\psi,m)$-stable rational
approximation of order $q$, and
let  $\varphi\in (0,\psi)$ and $\theta \in (\varphi, \psi)$.
Then there exists a constant $C = C(\theta)$ such that the following holds for all operators $A \in \operatorname{Sect}(\varphi)$ for some $\varphi\in [0,\theta)$:
\begin{enumerate}[\rm(i)]
\item 
If $|r(\infty)|=1$, $s \in [0,q+1]$,  and $\delta_s:=\min\{\msm,q\}$, then for all $x \in {\rm dom}(A^s)$, $t > 0$ and $n \in\N$,
\begin{equation}\label{ThAPHA}
\|e^{-tA}x-r^n(tA/n)x\| \le \frac{C M_\theta(A)t^s}{n^{\delta_s}} \|A^sx\|.
\end{equation}
\item  
If $|r(\infty)|<1$ and $s \in [0,q+1]$, then for all $x \in {\rm dom}(A^s)$, $t > 0$ and $n\in\N$,
\begin{equation}\label{gammaRA}
\|e^{-tA}x-r^n(tA/n)x\|\le \frac{C M_\theta(A) t^s}{n^q}\|A^sx\|.
\end{equation} 
\end{enumerate}
\end{thm}

\begin{proof}
 Replacing $A$ by $tA$ for $t>0$, we may assume that $t=1$.
   For $s\in [0,q+1]$, let $g_s(z) = z^{-s}$ and
 \[
\Delta_{n}(z):=e^{-z}-r_n(z), \qquad z \in \Sigma_\psi.
\]
Then $ g_s \Delta_{n} = \Delta_{n,s} \in \H_\psi$.

  For $\ep>0$, the extended holomorphic calculus can be applied to the invertible operator $A+\ep$.
    In that calculus, $g_s(A+\ep) = (A+\ep)^{-s}$ in the standard theory of fractional powers.
  Moreover, $g_s$ is in the domain of the extended $\H_\psi$-calculus for $A+\ep$,
  and $\Delta_{n}(A+\ep) = e^{-(A+\ep)} - r_n(A+\ep)$ in both calculi.  Then
\[
\Delta_{n,s}(A+\ep) = (A+\ep)^{-s} \big( e^{-(A+\ep)} - r_n(A+\ep) \big)
\]
in the extended holomorphic calculus.   Since the two calculi are compatible (see \cite[Proposition 8.1(iv)]{BaGoTo_J}) this also holds in the $\H_\psi$-calculus.   Hence, for $x \in \dom{(A^s)} = \dom((A+\ep)^s)$,
\begin{align*}
\big\| \big( e^{-(A+\ep)} - r_n(A+\ep) \big)x \big\| &= \left\|\Delta_{n,s}(A+\ep) (A+\ep)^sx \right\| \\
&\le  \|\Delta_{n,s}(A+\ep)\|\, \|(A+\ep)^s x\|.
\end{align*}
If $s>0$, then we can use the estimates from Theorem \ref{T1PH} and \eqref{estimate} to obtain \eqref{ThAPHA} or \eqref{gammaRA} with $A$ replaced by $A+\ep$.   Letting $\ep\to0$, and using the relation that $\lim_{\ep\to0} M_\theta(A+\ep) = M_\theta(A)$ and \cite[Proposition 3.1.9 c)]{Haase_b}, we obtain \eqref{T1PHA} or \eqref{gammaRA} for $A$.   For $s=0$, $\Delta_{n,0}(\infty) = - r(\infty)$ and there is an additional term $|r(\infty)|^n$ in the estimate for $\|\Delta_{n,0}(A+\ep)\|$.   This can be covered by changing the constant $C$.
\end{proof}

\begin{rem} \label{Tt}
The case $s=0$ of Theorem \ref{T1PHA}(ii) shows that for every $t \ge 0$,
 \[
 \lim_{n\to\infty} \|e^{-tA} - r^n(tA/n)\| = 0.
 \]
Let $\rho_1(z) = (z+1)^{-1}$.
Then $\rho_1$ is $\mathcal{A}(\pi/2,1)$-stable rational approximation of order $1$, with $\rho_1(\infty)=0$.
If $A \in {\rm Sect(\pi/2)}$ then $-A$ generates a $C_0$-semigroup $(T(t))_{t\ge0}$, and it is known that $\rho_1^n(tA/n)$ converges strongly to $T(t)$ for each $t \ge 0$ (\cite[Corollary III.5.5]{EN}).
Hence $T(t) = e^{-tA}$.
\end{rem}

The following corollary provides a natural example of a function $r$ satisfying the assumptions of Theorem \ref{T1PHA}(i).

\begin{cor}\label{Pade}
Let $r=P_k/Q_k$, $k\in \N$, be the diagonal Pad\'e approximation of order $2k$ to
$e^{-z}$, and let $A\in \Sect(\varphi),\varphi\in [0,\pi/2)$ and  $\theta \in (\varphi, \pi/2)$.
Then
there exists $C=C(\theta)$ such that, for all $s\in [0,q+1]$, $t\ge0$, $n\in \N$ and $x \in {\rm dom}(A^s)$, 
\begin{equation}\label{Pade1}
\|e^{-tA}x-r^n(tA/n)x\| \le \frac{CM_\theta(A)t^s}{n^{\gamma_s}} \|A^sx\|,
\end{equation}
where $\gamma_s=2\min\{s,k\}$.
\end{cor}

\begin{proof}
By  \cite[Theorem 4.12]{Hairer} (or \cite[Theorem 7]{Wanner}), the diagonal Pad\'e approximations of $e^{-z}$ are $\mathcal A(\pi/2)$-stable.  By \cite[Theorem 3.12]{Hairer}, $P_k$ and $Q_k$ have the same degree, $P_k(z)=Q_k(-z)$ for all $z \in \mathbb C,$ and $P_k$  contain the term $z^{k-1}$ with a non-zero coefficient.
Thus $r$ satisfies \eqref{order} with $m=1$, that is, they are $\mathcal A (\pi/2, 1)$-stable.  Now \eqref{Pade1} follows directly from \eqref{ThAPHA}.   
\end{proof}

\section{Optimality in terms of the algebra $\mathcal{H}_\psi$}  \label{sect7}

In this and the next section we address the optimality of Theorem \ref{T1PHA},
and show that the approximation rates provided by this theorem are sharp.
 We define
\begin{equation}
\Delta_{n,s,1}(z):=\Delta_{n,s}(z+1), \qquad z \in \C_+.
\end{equation}

\begin{thm}\label{T1PH11}
Let  $r$ be an $\mathcal A(\psi, m)$-stable rational approximation. Then the following hold.
\begin{itemize}
\item [\rm{(i)}] If $r$ is of exact order $q$,
then there exists $C>0$
such that, for all $\theta\in (0,\psi)$, $s\in [0,q+1]$ and $n\in\N$,
\begin{equation}\label{gammaR1}
\|\Delta_{n,s}\|_{\H_{\theta,0}} \ge|\Delta_{n,s}(1)| \ge \frac{C}{n^q}.
\end{equation}
\item [\rm{(ii)}] If  $r$ is of order $q$ and
$|r(\infty)|=1$, then
there exists $C=C(\theta)$ such that, for all $\theta\in(0,\psi)$, $s \in [0,q +1]$ and $n\in\N$,
\begin{equation}\label{ThAPH1}
\|\Delta_{n,s,1}\|_{{\mathcal H}_{\theta,0}}\ge \frac{C}{n^{\msm}}.
\end{equation}
\end{itemize}
\end{thm}

\begin{proof}  
The first inequality in \eqref{gammaR1} is immediate from \eqref{11inf},
so we only need to prove the second inequality.
Since $r$ is of exact order $q$, there exists  $a\ne0$ such that
\[
r(z) = e^{-z} + a z^{q+1} + \mbox{O}(|z|^{q+2}), \qquad z \to 0.
\]
Let
\begin{align*}
g(z) &:= r(z)e^z - 1 - a z^{q+1} = \mbox{O}(|z|^{q+2}), \\
h(z) &:= r(z)e^z-1 = g(z)+az^{q+1} = \mbox{O}(|z|^{q+1}).
\end{align*}
There exist $R>0$ and $c$  such that $g$ and $h$ are holomorphic in $\mathbb D_R$ and
\begin{equation}\label{m}
|g(z)|\le c|z|^{q+2},\qquad
|h(z)|\le \min\{\tfrac{1}{2}, c|z|^{q+1}\},\qquad z\in \mathbb D_R.
\end{equation}

Let $z \in \mathbb{D}_R$ and
\[
\eta(z) = \log(1+h(z)) - a z^{q+1}, \quad z \in \mathbb{D}_R.
\]
Using the elementary inequality
\[
|\log(1+\zeta)-\zeta|\le |\zeta|^2,\qquad |\zeta|\le 1/2,
\]
for $\zeta=h(z)$, we obtain
\[
|\eta(z)|\le |h(z)|^2 + |g(z)| \le c_1|z|^{q+2},
\]
and
\[
r(z)= e^{-z}(1+h(z)) = e^{-z}e^{az^{q+1}}e^{\eta(z)}.
\]

Now let $z \in \C$ and $n \ge |z|/R$,
so that $z/n \in \mathbb{D}_R$.   Then
\begin{equation}\label{ABC}
\Delta_{n,s}(z)=\frac{e^{-z}-r_n(z)}{z^s}
= \frac{e^{-z}(1-e^{az^{q+1}/n^q}e^{\eta_n(z)})}{z^s},
\end{equation}
where
\[
|\eta_n(z)|=n|\eta(z/n)|\le  c_1\frac{|z|^{q+2}}{n^{q+1}},\qquad
z\in \mathbb D_{Rn}.
\]
Setting $z=1$ in \eqref{ABC} and letting $n\to\infty$,
we see that $ \eta_n(1) = \mbox{O}(n^{-(q+1)})$, and
we infer that
\begin{align*}
|\Delta_{n,s}(1)|
= e^{-1}|1-\exp \left(a/n^q + \eta_n(1)\right)|
= \frac{|a|}{en^q}+\mbox{O}(n^{-(q+1)}), \quad n \to \infty,
\end{align*}
uniformly for $s \in [0,q+1]$.
Thus \eqref{gammaR1} holds.

Now we will prove \eqref{ThAPH1} under the additional assumption that $|r(\infty)|=1$. 
Let $s \in [0,q+1]$ and $n \in \mathbb N$. Using  \eqref{D-r} with $\epsilon=1$, we note that
\[
\|\Delta_{n,s,1}\|_{{\mathcal H}_{\theta,0}}\ge
Q_1^1(\theta,n,s) - \frac{q+2}{\cos\theta} e^{-Rn\cos\theta}.
\]
Taking into account Remark \ref{rem_eps} and applying the bound \eqref{QQAEps} of  Lemma \ref{Ac1} with $T=q+1$
and $r_1(z):=r(z+1)$ in place of $r$,
we infer that there exists $C=C(\theta)>0$
such that
\[
\|\Delta_{n,s,1}\|_{{\mathcal H}_{\theta,0}}\ge \frac{C}{n^{\msm}}.
\]
Thus \eqref{ThAPH1} holds.
\end{proof}

For $\gamma \in (0,1)$, $s \ge 0$, and $n\in\N$, let 
\[
{\Delta}^{[\gamma]}_{n,s,1}(z):={\Delta}_{n,s,1}(z^\gamma),
\qquad z\in \C_+.
\]

\begin{cor}\label{Cnew}
Let $r$ be an $\mathcal A(\psi, m)$-stable rational approximation of order $q$, with $|r(\infty)|=1$, and
let $\psi\in (0,\pi/2]$, $\theta\in(0,\psi)$, $\gamma:={2\theta}/{\pi} \in (0,1)$,
 $s \in [0,q+1]$ and $n \in \N$.
Then  $\Delta^{[\gamma]}_{n,s,1} \in \mathcal{LM}$.
Moreover, there is a constant $C = C(\theta)>0$ such that, for all $s \in [0,q+1]$ and $n\in\N$,
\begin{equation} \label{Cn1}
\|\Delta^{[\gamma]}_{n,s,1}\|_{\mathcal{LM}} \ge \frac{C}{n^{\delta_s}},
\end{equation}
where $\delta_s=\min\,\{\msm,q\}$.
\end{cor}

\begin{proof}

Let $n \in \mathbb N$ and $s \in [0,q+1]$.  
By direct verification,
$\Delta_{n,s,1} \in \H_{\theta}$.  Then, by \cite[Lemma 4.9]{BaGoTo_J} $\Delta^{[\gamma]}_{n,s,1} \in \H_{\pi/2}$,
and
by \cite[Theorem 4.12]{BaGoTo_J} we have $\Delta^{[\gamma]}_{n,s,1} \in \mathcal{LM}$.

Let $\varphi\in (0,\pi/2)$ and $\tilde \theta=\gamma\varphi\in (0,\theta)$.
Then, as noted in Section \ref{sec1},  $\Delta^{[\gamma]}_{n,s,1}$ belongs to $\H_\varphi$ and
 \eqref{lm-h} shows that
\[
\|\Delta^{[\gamma]}_{n,s,1}\|_{\mathcal{LM}}\ge \frac{\cos\varphi}{2}\|\Delta^{[\gamma]}_{n,s,1}\|_{\mathcal{H}_{\varphi,0}}
= \frac{\cos\varphi}{2}\|\Delta_{n,s,1}\|_{\mathcal{H}_{\tilde \theta,0}}.
\]
The assertion \eqref{Cn1} now follows from Theorem \ref{T1PH11}. 
 \end{proof}
 
 \begin{rem} \label{LL1}
Let  $\mathcal{L}L^1$ denote the space of Laplace transforms $\mathcal{L}f$ of functions $f$ in $L^1(\R_+)$.   This is a non-unital closed subalgebra of the Hille-Phillips algebra $\mathcal{LM}$.    Theorem 4.12 in \cite{BaGoTo_J} actually shows that $\Delta_{n,s,1}^{[\gamma]} \in \mathcal{L}L^1 + \C$, where $\C$ represents the constant functions on $\C_+$.   The space $\mathcal{L}L^1$ arises also in Section \ref{sect8}.
\end{rem}

\section{Optimality for semigroup generators} \label{sect8}

The estimates in Theorem \ref{T1PHA} for semigroup generators are optimal, as we will show in this section.
We start by showing that Theorem \ref{T1PHA}(ii) is optimal.  Given the estimate in Theorem \ref{T1PH11}(ii), the optimality of \eqref{gammaRA} is shown in the following trivial example.

\begin{exa} \label{opttr}
Let $r$ be an $\mathcal{A}(\psi)$-stable rational approximation of exact order $q$, 
and let $A$ be the identity operator $I$ on any Banach space $X$, so $A$ is sectorial of angle $0$.   
Then $A^s$ is the identity operator and $\Delta_{n,s}(A)$ is the operator of multiplication by the scalar $\Delta_{n,s}(1)$ on $X$.  
Thus it is immediate from Theorem \ref{T1PH11}(i) that there exists $C>0$ such that
\[
\|(e^{-A} - r^n(A/n))x\| = |\Delta_{n,s}(1)|\, \|A^s x\| \ge \frac{C}{n^q} \| x\|
\]
for all $s\in [0,q+1]$, $n \in \N$ and $x \in X$.  Thus the estimate \eqref{gammaRA} in Theorem \ref{T1PHA} is optimal for semigroup generators (up to a multiplicative constant).
\end{exa}

Next we will show that Theorem \ref{ThAPHA}(i) is optimal, by first considering a specific example of a semigroup of multiplication operators.

\begin{exa}  \label{optex}
For $f \in L^1(\R_+)$, let $\mathcal{L} f$ be the Laplace transform of $f$.  Consider the Banach space
\[
X=\mathcal{L}L_1 := \big\{\mathcal{L} f: f \in L^1(\R_+)\big\},
\]
with the norm given by
\begin{equation}\label{N}
\|\mathcal{L} f \|_{X}:=\|f\|_{L^1(\R_{+})},\qquad \mathcal{L} f \in X,
\end{equation}
cf.\ \eqref{hpn}.
For $t\ge 0$, let
\[
(T(t)\mathcal{L} f)(z)=e^{-tz}(\mathcal{L} f)(z),\qquad \mathcal{L} f\in X, \; z \in \C_+.
\]
The map $f \mapsto \mathcal{L} f$ is a surjective isometry from $L^1(\R_+)$ to $X$ which transforms the $C_0$-semigroup of right shifts on $L^1(\R_+)$ into the multiplication operators on $X$.
Let $-D$ be the generator of $(T(t))_{t\ge 0}$.
By the HP-calculus, for all $\mathcal{L} f \in X$ and $\mathcal{L}\mu\in\mathcal{LM}$,
\begin{align*}\label{P}
(\mathcal{L} \mu(D)\mathcal{L} f)(z)&=\int_0^\infty (T(t)\mathcal{L} f)(z) \,d\mu(t)\\
&=\int_0^\infty e^{-zt}\mathcal{L} f(z)\,
\,d\mu(t)=\mathcal{L} \mu(z)\mathcal{L} f(z)=\mathcal{L}(f * \mu)(z).
\end{align*}
Then, by \cite[Proposition 4.7.43(i)]{Dal} (cf.\ \eqref{hpn1}),
\begin{equation} \label{gnorm}
\|\mathcal{L} \mu(D)\|= \|\mathcal{L} \mu\|_{\mathcal{LM}}.
\end{equation}
For $\mathcal{L} f\in X$ and $z, \lambda\in\C_+$, 
\[
\big((\lambda+D)^{-1}\mathcal{L} f\,\big)(z)=\int_0^\infty e^{-\lambda t} e^{-zt}(\mathcal{L} f)(z)\,dt=
\frac{(\mathcal{L} f)(z)}{\lambda+z}.
\]
Hence $D$ is of the form
\begin{align}\label{defa}
{\rm dom}(D) &= \big\{\mathcal{L} f\in X:\, z(\mathcal{L}) f(z)\in X\big\},\\
(D\mathcal{L} f)(z) &= z (\mathcal{L} f)(z),\qquad z\in \C_{+}.\notag
\end{align}
The considerations above are variations of the elementary
transference principle considered in \cite[Example 5.2]{Haase}, for example.

Let $\theta \in [0,\pi/2)$, $\gamma=2\theta/\pi$ and $A:=D^\gamma$,
where $D$ is given by \eqref{defa}.
By the extended HP (or holomorphic) functional calculus,
$A \in \Sect(\theta)$, and
\[
(A f)(z)=z^\gamma f(z),\qquad f\in X,\quad z\in \C_{+}.
\]

Let $r$ be  an $\mathcal A(\psi, m)$-stable rational approximation of order $q$ with $|r(\infty)|=1$,
and let $s \in [0,q+1]$.
By the composition rule for the extended holomorphic calculus \cite[Theorem 2.4.2]{Haase},
\[
\Delta_{n,s,1}(A)={\Delta}^{[\gamma]}_{n,s,1}(D).
\]
It follows from Corollary \ref{Cnew} that
\begin{equation}\label{Opt1}
\|\Delta_{n,s,1}(A)\|   =  \|\Delta^{[\gamma]}_{n,s,1}(D)\| = \|\Delta^{[\gamma]}_{n,s,1}\|_{\mathcal{LM}} \ge \frac{C}{n^{\delta_s}}
\end{equation}
for some $C=C(\theta)>0$
and all $n \in \N$ and $s \in [0,q+1]$, where $\mathbf{\delta_s=}\min\,\{\msm,q\}$.
\end{exa}

Next we show that \eqref{ThAPHA} is optimal when $|r(\infty)|=1$ and $\msm<  q$. The case when $s  \in [mq/(m+1), q+1)$ has already been covered in Example \ref{opttr}.

\begin{thm}\label{optimal_t}
Let $r$ be an $A(\psi, m)$-stable rational approximation of order  $q$ with $|r(\infty)|=1$.
 Then for every  $(\alpha_n)_{n \ge 1} \subset (0,\infty)$ with $\lim_{n\to\infty} \alpha_n=0$
and all $\theta \in [0,\pi/2)$ and  $s \in [0, mq/(m + 1))$,
there exist a Banach space $\mathcal X$, an operator $\mathcal A \in \Sect (\theta)$ on $\mathcal X$,
a vector $x \in {\rm dom}(\mathcal A^s)$ and a constant $C=C(\theta)>0$ such that
\begin{equation}\label{lower}
\|e^{-\mathcal A}x-r^n(\mathcal A/n) x\|\ge C \alpha_n n^{-\msm}, \qquad n \in \mathbb N.
\end{equation}
\end{thm}

\begin{proof}
Without loss of generality, we can assume that $(\alpha_n)_{n \ge 1} \subset (0,1/2)$.
Let $\theta \in [0,\pi/2)$ be fixed, and let $X$ and $A \in \Sect(\theta)$
be as in Example \ref{optex}.

Define a Banach space $\mathcal{X}=c_0(X)$ by
 \begin{align*}
\mathcal{X}:=&\{(x_n)_{n\ge 1} \subset X: \,\lim_{n\to\infty}\,\|x_n\|=0\},\\
\|x\|_{\mathcal X}:=&\sup_{n\ge 1}\,\|x_n\|,\qquad
x=(x_n)_{n\ge 1}\in \mathcal{X},
\end{align*}
and consider the linear  operator
$\mathcal{A}$ on $\mathcal X$
given by
\begin{align*}
{\rm dom}(\mathcal{A}) &:=\{(x_n)_{n\ge 1}\in \mathcal X: x_n \in {\rm dom}(A), \,\, (A x_n)_{n\ge 1} \in \mathcal X \},\\
\mathcal{A}x &:=
((A+I) x_n)_{n\ge 1}, \qquad x=(x_n)_{n\ge 1} \in \rm{dom}(\mathcal A).
\end{align*}
It is straightforward to show that $\mathcal A \in \Sect (\theta)$.

Let $s \in  [0, mq/(m + 1))$.  
By Example \ref{optex} there exists $C=C(\theta)>0$ such that for all $n \in \mathbb N$,
\[
\|\Delta_{n, s, 1}(A)\|\ge {C}{n^{-\msm}}.
\]
Hence there exist $y_n\in X$,
satisfying
\[
\|y_n\|={\alpha_n},\qquad
\|\Delta_{n, s, 1} (A) y_n\|\ge 2^{-1} C  {\alpha_n} n^{-\msm}.
\]
Observe that by the spectral mapping theorem for fractional powers
\cite[Theorem 3.1.1 j)]{Haase_b}, the operators
$(A+I)^s$ and $\mathcal{A}^s$ are boundedly invertible.   
Moreover, the operators $\mathcal{A}^{-s}$ and $\Delta_n(\mathcal{A})$
act on $\mathcal{X}$ as diagonal copies of $(A+I)^{-s}$ and $\Delta_n(A+I) = \Delta_{n,1}(A)$ (see \eqref{f+e}).

Let  $y=(y_n)_{n \ge 1} \in \mathcal X$,  $x_n=(A+I)^{-s} y_n$ and 
\[
 x=\mathcal{A}^{-s}y=((A+I)^{-s}y_n)_{n\ge 1}.
\]
 Thus $x \in {\rm dom}(\mathcal A^s)$.
Using the product rule for
the (extended) holomorphic functional calculus,
we infer that, for every $n\in\N$,
\begin{multline*}
\|\Delta_{n} (\mathcal A)x\|_{\mathcal X}\ge 
 \|\Delta_{n}(A+I) x_n\|=\|\Delta_{n,1}(A) x_n\| \\
 =\|\Delta_{n, s, 1} (A)y_n\|
\ge 2^{-1} C  \alpha_n n^{-\msm}. \qedhere
\end{multline*}
\end{proof}

\begin{rem}
The Banach space $\mathcal{X}$ in Theorem \ref{optimal_t} is the same space, independent of $r$, $(\alpha_n)$, $\theta$ and $s$, and the operator $\mathcal{A}$ depends only on $\theta$. 
\end{rem}

\section{Appendix: Proof of Lemma \ref{asymptotic}}  \label{AppA}

We recall Lemma \ref{asymptotic} and give a proof.   For $\psi\in(0,\pi/2]$ and $R>0$, let
\[
\Sigma_\psi(R):= \{z \in \Sigma_\psi : |z| <R\}.
\]

We show that if  $r$ is an $\mathcal{A}(\psi)$-stable rational approximation of order $q$, and $\theta \in (0, \psi)$, $R>0$,  then there exist $\a > 0$ and
$C=C(\theta,R)$ such that
\begin{equation}\label{Main112}
\left|\Delta_{n, s}'(z) \right|\le
C(q+1-s+|z|)\frac{|z|^{q-s}}{n^q}e^{-\a|z|},
\end{equation}
for all $z \in \Sigma_\theta(Rn)$, $s \in [0, q+1]$ and $n \in \mathbb N$.

To this aim, let $\theta \in (0,\psi)$ and $R>0$ be fixed.   Constants that appear may depend on $r$, $\theta$ and $R$ but not on $s$ or $n$.

For $z\in \Sigma_\psi$ and $n \in \mathbb N$, define
\[
F_{n,s}(z):=\frac{e^{-nz}-r^n(z)}{z^s},
\]
and note that
\[
\Delta'_{n,s}(z)=\frac{1}{n^{s+1}}F'_{n,s}(z/n).
\]
Thus the estimate (\ref{Main112}) is equivalent to the estimate
\begin{equation}\label{FF}
\left|F'_{n,s}(z)\right|\le C n((q+1-s)+n|z|)|z|^{q-s}e^{-\a n|z|}
\end{equation}
for some $C, \a >0$ (with the same values as in (\ref{Main112})) and  all $z\in \Sigma_\theta(R)$, $s \in [0,q+1]$
and $n \in \mathbb N$.

To establish \eqref{FF} we need several auxiliary bounds.
By (\ref{approxim}), there exists $a \in \C$ such that
\begin{equation}\label{M1}
r(z)-e^{-z}=az^{q+1}+g_1(z), \qquad z\in\Sigma_\psi,
\end{equation}
where
\[
|g_1(z)|\le M_0|z|^{q+2}\qquad \text{and}\qquad
|g_1'(z)|\le M_1|z|^{q+1}
\]
for some $M_0, M_1>0$.
In particular,
\begin{equation}\label{RR}
r'(z)+r(z)
=a(q+1)z^{q}+ h(z),\qquad
|h(z)|\le M_2|z|^{q+1}
\end{equation}
for some $M_2>0$.

By
\cite[Lemmas 9.2 and 9.4]{Thomee_b},
for every $R>0$ there exist  $\a>0$
and $C_1>0$ (depending on $\theta$ and $R$) such  that
\begin{align}
|r(z)| &\le e^{-\a|z|},\label{c01}\\
|r^k(z)-e^{-kz}|&\le C_1 n |z|^{q+1}e^{-\a k|z|},\label{c00}
\end{align}
for all $z\in \Sigma_\theta(R)$ and $k \in \mathbb N$.
To deduce \eqref{FF}, we will derive a strengthened version of (\ref{c00}).
Namely, we will show that, for all $z\in \Sigma_\theta(R)$ and $n \in\N$,
\begin{equation}\label{So1}
r^n(z)-e^{-nz}=anz^{q+1}r^{n-1}(z)+ g_n(z),
\end{equation}
with
\begin{equation}\label{M3d}
|g_n(z)|\le
M_3 n^2|z|^{q+2}e^{-\a n|z|}
\end{equation}
for some $M_3>0$.

If $z\in \Sigma_\theta(R)$, then
\begin{multline*}
\null\hskip30pt (r^n(z)-e^{-nz})-n(r(z)-e^{-z})r^{n-1}(z)\\
=(r(z)-e^{-z})\sum_{j=0}^{n-1}r^j(z)(e^{-(n-1-j)z}-r^{n-j-1}(z)).\hskip30pt
\end{multline*}
Hence
\[
g_n(z) = (r(z)-e^{-z})\sum_{j=0}^{n-1}r^j(z)(e^{-(n-1-j)z}-r^{n-j-1}(z))   + ng_1(z)r^{n-1}(z).
\]
Using \eqref{c01}, \eqref{c00} (for $k = 1,2,\dots,n-1$) and  (\ref{M1}), we obtain
\begin{align*}
|g_n(z)|
&\le C_1 |z|^{q+1} e^{-\a|z|}
 \sum_{j=0}^{n-1} e^{-\a j|z|} C_1(n-j-1)|z|^{q+1}e^{-\a(n-j-1)|z|} \\
 & \null\hskip30pt + n M_0 |z|^{q+2} e^{-\a (n-1)|z|} \\
&\le \left( C_1^2  |z|^q  + M_0 e^{\a |z|}\right) n^2|z|^{q+2}e^{-\a n|z|} \le M_3 n^2|z|^{q+2}e^{-\a n|z|},
\end{align*}
where $M_3 = C_1^2  R^q + M_0 e^{\a R}$.

Now, observing that
\[
(e^{-nz}-r^n(z))'= n(r^n(z)-e^{-nz})-n\left(r'(z)+r(z)\right)r^{n-1}(z),
\]
and using \eqref{RR} and \eqref{So1}, we write
\begin{align*}
F'_{n,s}(z)
&=\frac{n(r^n(z) -e^{-nz})}{z^s}
-  \frac{nr^{n-1}(z)(r'(z)+r(z))}{z^s}
+ \frac{s(r^n(z)-e^{-nz})}{z^{s+1}} \\
&=\frac{n(r^n(z)-e^{-nz})}{z^s}
-na(q+1-s)z^{q-s}r^{n-1}(z)
\\
&\hskip30pt  \null -\frac{s g_n(z)}{z^{s+1}} - \frac{nr^{n-1}(z)h(z)}{z^s}.
\end{align*}
Using (\ref{RR}),
\eqref{c01}, \eqref{c00} (for $k=n$), (\ref{So1}) and \eqref{M3d},
we obtain that
\begin{align*}
|F'_{n,s}(z)|
&\le C_1 n^2|z|^{q+1-s}e^{-\a n|z|}
+n|a|(q+1-s)|z|^{q-s}e^{-\a(n-1)|z|}  \\
&\hskip20pt \null + M_3(q+1)n^2|z|^{q+1-s}e^{-\a n|z|}
+ M_2 |z|^{q+1-s}e^{-\a(n-1)|z|} \\
&\le Cn(q+1-s+n|z|)|z|^{q-s}e^{-\a n|z|},
\end{align*}
where
\[
C=\max \left(|a|, C_1 + M_3(q+1)R + M_2Re^{\a R}\right).
\]
Thus \eqref{FF} holds. $\hfill \hfill \hfill \qedhere$


\begin{thebibliography}{99}

\bibitem{Bakaev_b} N. Yu.\ Bakaev, \emph{Linear discrete parabolic problems,} North Holland, 2006.

\bibitem{Bakaev} N. Yu.\ Bakaev, \emph{Some problems of well-posedness of difference schemes
on non-uniform grids,} Zhurn. Vychisl. Mat.\ i Mat.\ Fiz. \textbf{33}(1993), 561--
577 (in Russian); transl.\ in: Comput.\ Math.\ Math.\ Phys.\ \textbf{33}
(1993), 511--524.

\bibitem{Bakaev96} N. Yu.\ Bakaev, \emph{On the bounds of approximations of holomorphic semigroups,}
 BIT \textbf{35}(1995), 605--608.

\bibitem{Bakaev98} N. Yu.\ Bakaev, \emph{On variable stepsize Runge-Kutta approximations of a
Cauchy problem for the evolution equation,} BIT \textbf{38}(1998), 462--485.

\bibitem{BakaevO} N. Bakaev and A. Ostermann, 
\emph{Long-term stability of variable stepsize approximations of semigroups,}
 Math. Comp. \textbf{71} (2002), 1545--1567.

\bibitem{BaGoTo_J}
C. Batty, A. Gomilko, and Yu. Tomilov,
\emph{Functional calculi for sectorial operators
and related function theory},
J. Inst.\ Math.\ Jussieu,
\textbf{22} (2023), 1383--1463.

\bibitem{BTW}
P. Brenner, V. Thom\'ee and L. Wahlbin,  \emph{Besov Spaces and Applications to Difference
Methods for Initial Value Problems,} Lecture Notes in Math. \textbf{434}, Springer, 1975.

\bibitem{BT}
P. Brenner and V. Thom\'ee,
\emph{On rational approximation of semi-groups},
 SIAM J. Numer. Anal.,
\textbf{16} (1979),  683--694.

\bibitem{Crou}  M. Crouzeix, S. Larsson, S. Piskarev, and V. Thom\'ee,
\emph{The stability of rational approximations of analytic semigroups,}
BIT \textbf{33} (1993),  74--84.

\bibitem{Dal} H.G. Dales, \emph{Banach algebras and automatic continuity}, Oxford Univ.\ Press, 2000.

\bibitem{EgertR}
M. Egert and J. Rozendaal,
\emph{Convergence of subdiagonal Pad\'e approximations of $C_0$-semigroups},
J. Evol.\ Equ., \textbf{13} (2013), 875--895.

\bibitem{EN} K-J. Engel and R. Nagel,  \emph{One-parameter semigroups for linear evolution equations}, Springer, 2000.

\bibitem{Flory} S. Flory, F. Neubrander, and L. Weis, \emph{
Consistency and stabilization of rational approximation schemes for $C_0$-semigroups,} in 
Evolution equations: applications to physics, industry, life sciences and economics,
Progr. Nonlinear Differential Equations Appl., \textbf{55}, Birkh\"auser, Basel, 2003, 181--193.

\bibitem{Fujita} H. Fujita, N. Saito, and T. Suzuki, \emph{Operator theory and numerical methods}, Studies in Math. and its Appl., \textbf{30}, Elsevier, 2001.


\bibitem{GT}
 A. Gomilko and Yu.\ Tomilov,
 \emph{On convergence rates in approximation theory for operator semigroups},
J. Funct.\ Anal., \textbf{266} (2014), 3040--3082.


\bibitem{GKT}
A. Gomilko, S. Kosowicz, and Yu.\ Tomilov,
\emph{A general approach to approximation theory of operator semigroups},
J. Math. Pures Appl.  \textbf{127} (2019), 216--267.



\bibitem{Haase_b}  M. Haase, \emph{The functional calculus for sectorial operators,} Oper. Theory Adv. Appl., \textbf{169}, Birkh\"auser, Basel, 2006.

\bibitem{Haase} M. Haase, \emph{Semigroup theory via functional calculus,} T\"ubinger Berichte, \textbf{15} (2005/06), 195--208.

\bibitem{Hairer}
E. Hairer and G. Wanner, \emph{Solving ordinary differential equations, II,
Stiff and differential-algebraic problems,} Springer Series in Comput.\ Math., 
Second revised ed., \textbf{14}, Springer, Berlin, 2010.

\bibitem{Hansbo} A. Hansbo, \emph{Nonsmooth data error estimates for damped single step methods
for parabolic equations in Banach spaces,} Calcolo \textbf{36} (1999), 75--101.

\bibitem{HK}
R. Hersh and T. Kato, \emph{High-accuracy stable difference schemes for well-posed initial-value problems,}
SIAM J. Numer.\ Anal.,  \textbf{16} (1978), 670--682.

\bibitem{Jara1} P. Jara, F. Neubrander, and K. \" Ozer, \emph{Rational inversion of the Laplace transform,} J. Evol.\ Equ.\ \textbf{12} (2012),  435--457.

\bibitem{Kovacs}
M. Kov\'acs, \emph{On the convergence of rational approximations of semigroups on intermediate spaces},
Math.\  Comp.\ \textbf{76} (2007), 273--286.

\bibitem{Larsson} S. Larsson, V. Thomee, and L.B. Wahlbin,
\emph{Finite element methods for a strongly damped
wave equation}, IMA J. Numer.\ Anal.\ \textbf{11} (1991), 115--142.

\bibitem{Roux} M.-N. Le Roux, \emph{Semidiscretization in time for parabolic problems,}  Math.\ Comp.\ \textbf{33} (1979),  919--931.

\bibitem{Palencia1} C. Palencia, \emph{A stability result for sectorial operators in Banach spaces,} SIAM J. Numer.\ Anal.\ \textbf{30} (1993), 1373--1384.

\bibitem{Palencia} C. Palencia, \emph{On the stability of variable stepsize rational approximations of holomorphic semigroups,} Math.\ Comp.\ \textbf{62}(1994), 93--103.

\bibitem{Sobolevskii} P.  Sobolevskii and
and Huang Van Lai, \emph{Difference schemes of optimal type for approximation of solutions
of parabolic equations,} in Differential and Integro-Differential Equations, Novosibirsk, 1977, 37--43
(in Russian).

\bibitem{Thomee_b} V. Thom\'ee, \emph{Galerkin finite element methods for parabolic problems,} Second ed.,
 Springer Series in Comp.\ Math., \textbf{25}, Springer, Berlin, 2006.


\bibitem{Wanner}
G. Wanner, E. Hairer, and S. P. N\"orsett,
\emph{Order stars and stability theorems,}
BIT, {\bf 18} (1978), 475--489.

\bibitem{Yan}
 Y. Yan, \emph{Smoothing $A$ properties and approximation of time derivatives for parabolic equations: variable time steps,} BIT \textbf{43} (2003),  647--669.

\bibitem{Yan1}
Y. Yan, \emph{Smoothing properties and approximation of time derivatives for parabolic equations: constant time steps,} IMA J. Numer.\ Anal.\ \textbf{23} (2003), 465--487.


\end{thebibliography}
\end{document}